\documentclass[11pt]{article} 
\usepackage{natbib}
\usepackage{graphicx}
\usepackage{adjustbox}
\usepackage[utf8]{inputenc}
\usepackage{amsmath,amsthm,amssymb,cases,booktabs,enumitem,siunitx}
\usepackage{multirow}
\usepackage{setspace}
\usepackage{todonotes, verbatim}
\usepackage[colorlinks,
            citecolor=blue,
            urlcolor=blue,
            linkcolor=blue]{hyperref}
\usepackage{doi}
\usepackage[noend]{algpseudocode}
\usepackage{algorithm}
\usepackage{authblk}

\usepackage{subcaption}

\usepackage[english]{babel}
\usepackage[T1]{fontenc}

\usepackage{color}
\definecolor{linkcol}{rgb}{0,0,0}
\definecolor{citecol}{rgb}{0,0,0.7}

\usepackage{pgfplots}
\pgfplotsset{compat=1.17}

\theoremstyle{plain}
\newtheorem{thm}{Theorem}
\newtheorem{proposition}[thm]{Proposition}

\theoremstyle{definition}
\newtheorem{definition}[thm]{Definition}

\newtheorem{example}[thm]{Example}

\def\C{\mathcal{C}}

\newcommand{\RCFP}{$\mathbb{R}^2$-CFP}
\newcommand{\rCFP}{$\mathbb{R}$-CFP}
\newcommand{\PWL}{piecewise linear}
\newcommand{\CFP}{Corridor Fitting Problem}
\newcommand{\RmCFP}{$\mathbb{R}^m$-CFP}

\newcommand{\iLinA}[1]{iLinA2D}

\newcommand{\define}{\mathrel{{\mathop:}{=}}}

\def\restriction#1#2{\mathchoice
              {\setbox1\hbox{${\displaystyle #1}_{\scriptstyle #2}$}
              \restrictionaux{#1}{#2}}
              {\setbox1\hbox{${\textstyle #1}_{\scriptstyle #2}$}
              \restrictionaux{#1}{#2}}
              {\setbox1\hbox{${\scriptstyle #1}_{\scriptscriptstyle #2}$}
              \restrictionaux{#1}{#2}}
              {\setbox1\hbox{${\scriptscriptstyle #1}_{\scriptscriptstyle #2}$}
              \restrictionaux{#1}{#2}}}
\def\restrictionaux#1#2{{#1\,\smash{\vrule height .8\ht1 depth .85\dp1}}_{\,#2}} 

\newenvironment{lcases}
  {\left\lbrace\begin{aligned}}
  {\end{aligned}\right.}

\allowdisplaybreaks

\newcommand{\relaxRonestrong}{($\mathcal{R}^{\textrm{DCFP}}$)}
\newcommand{\relaxRonestrongdecision}{($\mathcal{D}^{\textrm{DCFP}}$)}
\newcommand{\algmilp}{($\mathcal{A}^{\textrm{DCFP}}$)}
\newcommand{\relaxRone}{($\mathcal{R}^{\textrm{color}}$)}
\newcommand{\algcolor}{($\mathcal{A}^{\textrm{color}}$)}
\newcommand{\relmaximum}{($\mathcal{R}^{\textrm{MumC}}$)}
\newcommand{\algmaximum}{($\mathcal{A}^{\textrm{MumC}}$)}
\newcommand{\relmaximal}{($\mathcal{R}^{\textrm{MalC}}$)}
\newcommand{\algmaximal}{($\mathcal{A}^{\textrm{MalC}}$)}

\newcommand{\relaxR}{$(\mathcal{R})$}

\author[1]{Aloïs Duguet \thanks{Corresponding author: duguet@uni-trier.de}}
\author[2]{Sandra Ulrich Ngueveu}
\author[2]{François Lamothe}

\affil[1]{Trier University, Department of Mathematics, Universitätsring 15, Trier 54296, Germany}
\affil[2]{Université de Toulouse, INP, LAAS-CNRS, Toulouse, France}

\title{On the Minimum Number of Linear Pieces Required to Approximate Nonlinear Functions under an Accuracy Constraint}

\date{\today}

\begin{document}
\maketitle

\begin{abstract}

The approximation of nonlinear functions by \PWL{} functions is a tool commonly used when dealing with mixed-integer nonlinear problems. Typically, by replacing nonlinearities by \PWL{} functions one can transform the problem into a mixed-integer linear problem, which may be substantially easier to solve. 
However, using approximate functions can produce solutions that are infeasible for the original problem or far from optimal. To control these errors it is useful to bound the error created during the function approximation process. Moreover, obtaining a \PWL{} function with few pieces usually results in an easier to solve mixed-integer linear problem. 
This leads us to study the Corridor Fitting Problem. It consists in building a \PWL{} function with the minimum number of pieces which approximates a nonlinear function given a bound on the approximation error on each point of the domain.

The Corridor Fitting Problem has primarily been addressed for univariate functions or via heuristic approaches for multivariate functions. Notably, for the latter setting, no exact algorithms or established relaxations are currently known. In this work, we explore this aspect and propose exploitable relaxations of the Corridor Fitting Problem in $\mathbb{R}^m$ based on a discretization of the domain. We show that a structure of hypergraph coloring problem is induced by the discretization of the domain. We define four relaxations making use of this hypergraph coloring problem. We provide new best upper bounds for the classical instance set in $\mathbb R^2$ and we derive the first lower bounds for these instances, closing more than a third of the instances from the literature.
\end{abstract}


\section{Introduction}

Mixed-integer nonlinear programming (MINLP) is a wide class of optimization problems capable of modeling a large variety of real-world applications \citep{Borghetti08,Medeiros22,Hante20gas}, but the presence of nonlinear functions and discrete variables \citep{Belotti13} make them hard to solve in the general case. A common solution method consists in replacing each nonlinear function in the problem by a piecewise linear function \citep{Geissler12} to obtain a mixed-integer linear program (MILP) usually faster to solve than the original MINLP. However, replacing the nonlinear functions by \PWL{} functions without taking precautions results in solutions that can (i) be infeasible for the original problem, and (ii) have an optimal value unreasonably far from the optimal value of the original problem. 
Using piecewise linear functions satisfying a predefined approximation error allows to avoid these drawbacks under some conditions. For example, it is easy to check the following property: if only the objective function is nonlinear and it is replaced by a \PWL{} one at most $\delta$ away from it, then the optimal value of the approximated MILP is also at most $\delta$ away from the optimal value of the original MINLP.

The solution time of MILPs correlates with their size. Additionally, the number of constraints, continuous variables and binary variables in the MILP formulation of a \PWL{} function increases with the number of pieces. Therefore it is beneficial to ensure that the piecewise linear approximation at most $\delta$ away uses the least number of pieces.
The construction of such \PWL{} functions has been the subject of a number of articles which led to the formalization of the \textit{Corridor Fitting Problem} in \citet{Codsi25}. The present work is concerned with the improvement of methods to solve this problem.
For successful applications of \PWL{} approximations to MINLPs, we refer the reader to \citet{Rovatti14} and \citet{Silva14}, for example.

The Corridor Fitting Problem consists in building a \PWL{} function with the minimum number of pieces which approximates a nonlinear function given a predefined bound on the approximation error on each point of the domain. Exact algorithms exist for one-variable functions \citep{Rebennack19,Codsi25,warwicker21}. For two-variable functions exact methods are known only for some special cases such as the Euclidean Norm \citep{Duguet22b}. For the general case of $m$-variable functions, there are no exact methods or even relaxations. 

In this work, we address this gap and propose tailored relaxations for the general case of the Corridor Fitting Problem, then we focus on the two-variable case for which we improve two existing heuristics. The computational experiments are performed on the two-variable case. 

Our main contributions are the following.
\begin{itemize}
    \item We show that a discretization of the domain of the Corridor Fitting Problem induces the structure of a hypergraph coloring problem. Each edge of this hypergraph models that a subset of points of the domain cannot be covered by the same piece of a \PWL{} function that is a feasible solution to the Corridor Fitting Problem.
    \item From the hypergraph coloring structure we derive a hierarchy of four successive relaxations based on hypergraph coloring or clique problems. Each subsequent relaxation discards more of the original problem’s structure to enhance computational tractability.
    \item We implement the lower bounding algorithms for the two-variable case and we establish, for the first time, lower bounds for the classical two-variable instances. Results show the computational usefulness of the hypergraph coloring constraints. More specifically, the more information a relaxation uses, the better the results are on average.
    \item We provide new best feasible solutions for some of the classical instances for two-variable functions. These solutions were achieved by introducing improvements to the existing state-of-the-art heuristics.
    \item Finally, we close 16 out of the 45 instances from the literature.
\end{itemize}

The remainder of the paper is organized as follows. In Section \ref{sec:problem_def_sota} we formally define the Corridor Fitting Problem and provide a literature review. In Section \ref{sec:relaxations}, we introduce four subsequent relaxations of the Corridor Fitting Problem based on a hypergraph coloring problem that emerges from a domain discretization. We explain in Section \ref{sec:compute_edges} how to build the hypergraph coloring problem and how the lower bounding algorithms are derived from the relaxations. We introduce improvements to two state-of-the-art methods for computing feasible solutions of the Corridor Fitting Problem in the two-variable case in Section \ref{sec:new_UB}.
Finally, in Section \ref{sec:experiences_numeriques} we provide numerical results for the computation of lower bounds of the Corridor Fitting Problem in the two-variable case and compare them to the best known upper bounds to establish the optimality of feasible solutions for the first time. Section \ref{s:theconclusion} concludes the paper.

\section{Problem Definition and Literature Review}
\label{sec:problem_def_sota}
In this section, we formally define the problem tackled and provide a literature review.

\subsection{Problem Definition}

We consider a nonlinear function $f: D \subset \mathbb R^m \rightarrow \mathbb R$, with $D$ a polytope. Our goal is to build a \PWL{} function $g$ that satisfies the following constraints:
\begin{align}
    l(x) \leq g(x) \leq u(x) \quad \forall x \in D, \label{def_cont_corridor}
\end{align}
where the functions $u$ and $l$ are defined so that $g$ approximate $f$ to the needs. For example, for $g$ to approximate $f$ with a $\delta$-\textit{absolute error} the constraints \eqref{def_cont_corridor} read:
\begin{align}
    f(x) - \delta \leq g(x) \leq f(x) + \delta \quad \forall x \in D. \label{absolute_error}
\end{align}
Alternatively if $f$ is positive, for $g$ to approximate $f$ with an $\epsilon$-\textit{relative error}, the constraints \eqref{def_cont_corridor} read:
\begin{align}
    (1 - \epsilon) f(x) \leq g(x) \leq (1 + \epsilon) f(x) \quad \forall x \in D. \label{relative_error}
\end{align}

Constraints \eqref{def_cont_corridor} are called \textit{corridor constraints}, since they force the graph of the function $g$ to be inside a \textit{corridor} delimited by $u$ and $l$ as stated in the following definition. 

\begin{definition}[Corridor]
Let $D \subset \mathbb{R}^m$ be a full dimensional polytope and $u,l$ be two continuous functions from $D$ to $\mathbb{R}$ satisfying $l(x) < u(x)$ for all $x \in D$. The \textit{$\mathbb{R}^m$-corridor} $\C$ between $u$ and $l$ of domain $D$ is the set $\{(x,z) \in \mathbb{R}^{m+1} | x \in D,\; l(x) \leq z \leq u(x)\}$ and is denoted $\C = \mathit{Corridor}(u,l,D)$.
\end{definition}

Figure \ref{def_figure_corridors} shows two examples: an $\mathbb R$-corridor and an $\mathbb R^2$-corridor. In the left figure, the corridor is the space in gray in-between the two black curves representing functions $u$ and $l$. In the right figure, the corridor is the subspace of $\mathbb R^3$ in-between the two surfaces representing functions $u$ and $l$.

\begin{figure}[!ht]
\centering
\begin{subfigure}{.45\textwidth}
    \centering
    \includegraphics[width=\linewidth]{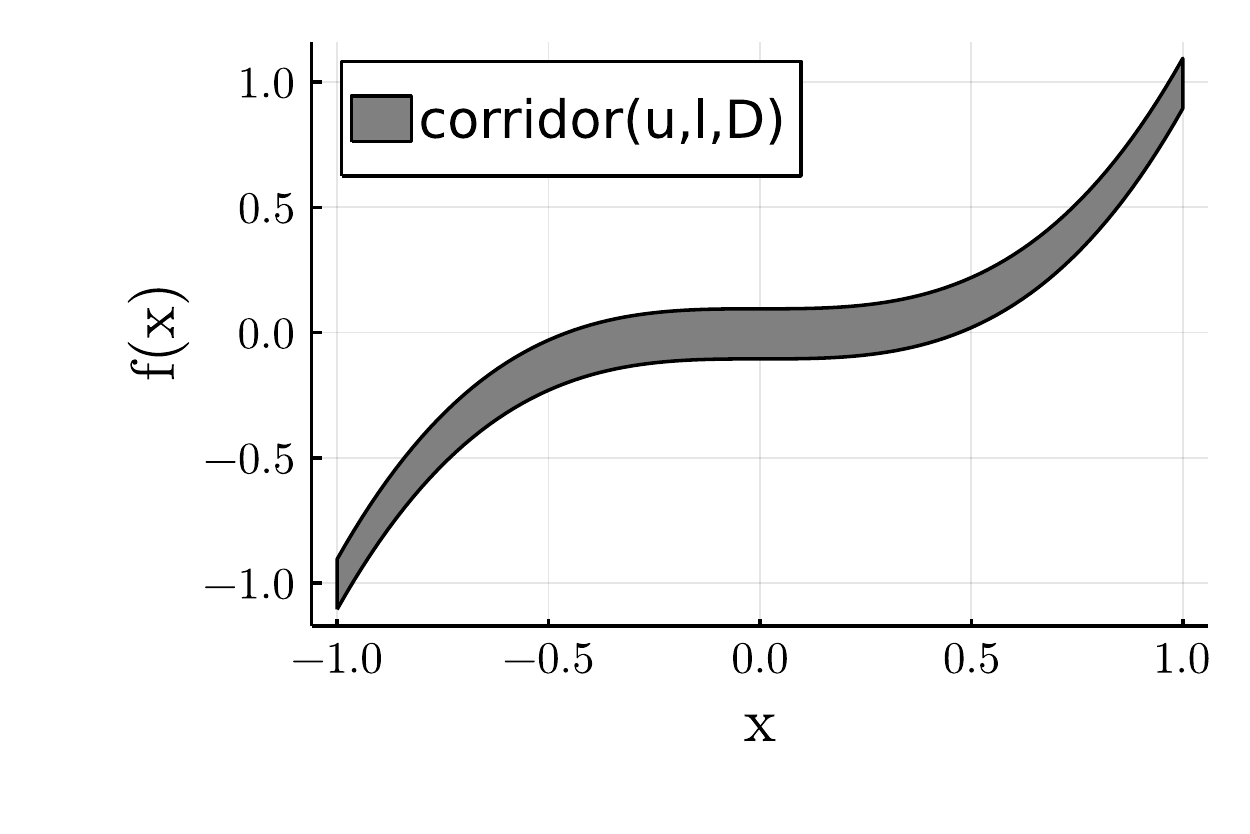}
    \caption{An $\mathbb R$-Corridor}
    \label{def_figure_corridor_1D}
\end{subfigure}%
\begin{subfigure}{.49\textwidth}
    \centering
    \includegraphics[width=\linewidth]{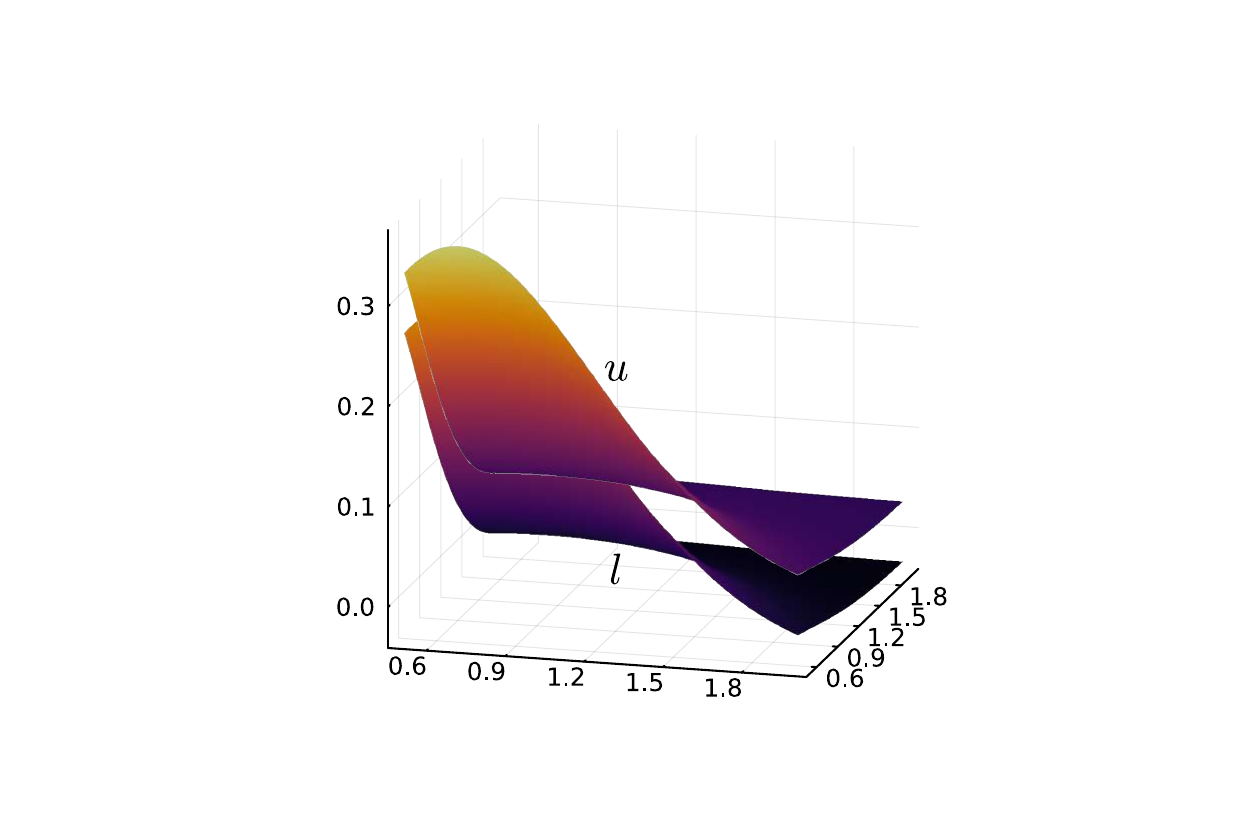}
    \caption{An $\mathbb R^2$-Corridor}
    \label{def_figure_corridor_2D}
\end{subfigure}
\caption{Examples of corridors in dimension one and two}
\label{def_figure_corridors}
\end{figure}

\begin{definition}[\PWL{} function] \label{def_pwl_function}
    Let $D \subset \mathbb R^m$ be a full dimensional polytope, $n$ a strictly positive integer, $(D_i)_{i = 1,...,n}$ a set of full dimensional polytopes of $\mathbb R^m$ covering $D$ and of empty intersection between their interiors.
    Let $(l_i)_{i = 1,...,n}:D_i \mapsto \mathbb R$ be linear functions. 
    We call $g$ such that for $x \in D$:
    \begin{align*}
        &g(x) = 
        \begin{lcases}
            &l_i(x) \quad &\text{if } x \in int(D_i), \\
            &\min_i l_i(x) \quad &\text{if } x \in bd(D_i),
        \end{lcases} \\
    \end{align*}
    a \PWL{} function with $n$ pieces, where $int(A)$ is the interior of the set $A$ and $bd(A)$ the boundary of $A$. The function $g$ is lower semicontinuous due to the minimization. 
    By replacing it by a maximization, $g$ is an upper semicontinuous function. 
\end{definition}
The lower semicontinuity is an important property when formulating a \PWL{} function $g$ in an MILP. Indeed, in this case $epi(g)$ is closed and thus it can be modeled as a union of polyhedra \citep{Vielma10}. 

We now formally introduce the corridor fitting problem, which consists in building a \PWL{} function with the minimum number of pieces for which its graph is included in a corridor $\C=\mathit{Corridor}(u,l,D)$:
\begin{definition}[\CFP{}] \label{def_CFP}
Let $D \in \mathbb R^m$ be a compact and full dimensional polytope and $u$ and $l$ be continuous functions from $D$ to $\mathbb R$ satisfying  $l(x) < u(x)$ for all $x \in D$. 
The problem of finding $g$, a \PWL{} function with $n$ pieces and domain $D$, solution of:

\begin{subequations}
\begin{numcases}{\mathit{(CFP)}\quad} \label{CFP}
        \underset{g}{min} \quad & $n$ \\
         \textit{s.t.} \quad & $l(x) \leq g(x) \leq u(x) \quad \forall x \in D \subset \mathbb{R}^m$ \label{constraint_corridor}
\end{numcases}
\end{subequations}
is called \textit{Corridor Fitting Problem in $\mathbb R^m$} or simply \RmCFP{} or CFP if the $m$ is clear from the context.
\end{definition}
Constraints \eqref{constraint_corridor} can be reformulated using the notion of corridor as follows: 
$$g(x) \in \mathit{Corridor}(u,l,D) \quad \forall x \in D.$$
The \RmCFP{} always admits a solution with a finite number of pieces because a continuous function defined on a compact domain of $\mathbb R^m$ can be approximated to arbitrary precision through a continuous \PWL{} function with a finite number of pieces \citep{Julian98}. However, it is a hard problem because there is an infinite number of constraints in Constraints \eqref{constraint_corridor}.

\subsection{Literature Review} 
\label{subsec:sota}

In this section, we describe the literature for the Corridor Fitting Problem in $\mathbb R$, $\mathbb R^2$, and $\mathbb R^m$ for $m \geq 3$.

\subsubsection{The $\mathbb R$-corridor fitting problem}

Most of the works consider the \rCFP{} with the additional constraint that the \PWL{} function is continuous and the corridor corresponds to an absolute function approximation, i.e., as in Constraints \eqref{absolute_error} \citep{Rebennack15b,Kong20,Rebennack19,Warwicker25}. For this particular case, the exact algorithms proposed are based on the same principle: compute the solution $p^*$ to a discretized version of the problem, i.e., by considering the corridor constraints only on a finite set of points of the domain, and then check if $p^*$ is also a feasible, thus optimal, solution to the problem with all the corridor constraints. If it is not feasible, then new corridor constraints on the points with the worst error are added. If it is optimal, then the algorithm stops with the optimal solution $p^*$. The key differences between the algorithms from the literature lie in the formulation and solution of the discretized problem solved at each iteration. Contrary to pre-existing exact algorithms that solved a MILP or a nonconvex NLP at each iteration \citep{Rebennack19,Kong20,Rebennack15b}, \citet{Warwicker25} uses the algorithm from \citet{Hiroshi86pwl} that has a linear complexity in the number of corridor constraints to satisfy. This resulted in more instances solved to optimality and a significant speedup with a computation time at least an order of magnitude lower.
Note that the discretized version of the \rCFP{} is equivalent to the problem of \PWL{} fitting of discrete data points \citep{Ploussard23}.

For the general case of the \rCFP{}, i.e to find an exact solution for (i) non-necessarily continuous (nnc) univariate \PWL{} functions and (ii) the corridor is not required to correspond to an absolute function approximation, \citet{Codsi25}, who introduced the term \emph{corridor}, proposes an algorithm in quasi-logarithmic time. It performs a dichotomy search on the domain for functions $u, l$ that do not present changes in concavity (i.e. that are either convex or concave). For general functions with changes in concavity, the authors show that one can benefit from the dichotomy search for most of the procedure, and execute the more expensive data fitting method only around the inflection points. The proposed method, LinA, demonstrates a substantial improvement in computational efficiency compared to the state-of-the-art, solving all benchmark instances in a negligible computation time. 

\subsubsection{The $\mathbb R^2$-corridor fitting problem}
\label{subsubsec:stateofartR2cfp}

No exact method exist for solving the $\mathbb R^2$-corridor fitting problem in the general case. Exact algorithms are known for some special cases such as the relative approximation \eqref{relative_error} of the euclidean norm on the domain $\mathbb R^2$ \citep{Duguet22b} and approximation heuristics are known for special cases such as the absolute approximation \eqref{absolute_error} of the nonlinear function $f(x,y) = xy$ with additional constraints on the piecewise linear function \citep{Barmann23}. For the general case, heuristics have been proposed \citep{Rebennack15a,Kazda21,Duguet22a,Kazda23} for which, to the best of our knowledge, there is no published work quantifying the deviation from the optimum, of number of pieces of the approximations obtained. 
We now proceed to describe those heuristics, which will be used as a benchmark to evaluate the lower and upper bounds we propose for the $\mathbb R^{m\geq 2}$-corridor fitting problem. All of them are developed for the absolute approximation case \eqref{absolute_error}, but most of the theory can be adapted to the general case of the \RCFP{} with a straightforward change.

\citet{Rebennack15a} describes two heuristic algorithms. The first is capable of producing continuous or nnc \PWL{} functions while the second can only produce continuous \PWL{} functions. The first heuristic denoted RK1D can be used if the contribution of the two variables in the function can be separated in two one-variable functions, be it linearly or nonlinearly. In this case, an algorithm finds the two optimal continuous one-variable PWL functions and combines them to build a single two-variable PWL function. This approach can be efficient for linearly separable functions, but is much less efficient for other functions. It generates \PWL{} functions where the domain of each piece is rectangular.

The second heuristic of \citet{Rebennack15a}, denoted RK2D, produces \PWL{} functions for which the domain of the pieces is a triangulation. The method initializes by dividing the rectangular domain by its diagonal into two triangles. Then, for each of these triangles, the algorithm solves an NLP to search for a linear function that respects the approximation error. This function is sought among a finite number of linear functions corresponding to the convex combination of the three vertices of the triangle of the form $(x_1, x_2, f(x_1, x_2) + s)$ where $s$ is a value in $[-\delta, \delta]$ representing a shift relative to the value of the function. If a solution to the NLP is found, the corresponding function is added to the PWL under construction. Otherwise, the triangular domain is subdivided. RK2D considers two possible refinement procedures: subdivision into three triangles around the point where the approximation error is maximum, or subdivision into four triangles, known in the literature as red refinement. Once the new domains have been obtained, RK2D processes them separately, searching for each one, if there is a linear function respecting the approximation error on the domain. This heuristic has the drawback of not being parsimonious with the number of pieces of the solution, as triangles are split into at least 3 triangles. 

\citet{Kazda21} proposes a different approach, which we refer to as KL21. Unlike RK2D, which generates linear pieces iteratively, KL21 computes a complete piecewise linear function over the entire domain at each iteration by solving a MILP that verifies approximation errors at a finite set of points in the domain of the nonlinear function. Then, two NLPs are solved for each linear piece to compute the points of maximum difference between the \PWL{} function and the nonlinear function. It is used to verify whether the approximation error is respected at every point of the continuous domain or not. If yes the algorithm stops. Otherwise, a new iteration of the algorithm is executed, which consists of adding new points to the set of points, and solving the resulting MILP to produce a new \PWL{} function. Another specificity of the heuristic of \citet{Kazda21} is that the desired \PWL{} function $g$ is continuous and is modeled within the MILP as a \textit{difference of convex continuous piecewise linear function} $g(x) = g^+(x) - g^-(x)$. As a result, the shape of each piece can be any convex polygon, not only triangles, which can lead to fewer pieces. However, the increasing size of the MILP problems to be solved imposes a significant limitation on the use of the heuristic for medium-sized or large instances. 

In \citet{Kazda23}, the KL21 method is improved with the aim of reducing computation time, at the expense of the number of pieces of the PWL solution functions. To achieve this, the iteration for solving MILP problems is replaced by a two-step iteration: the first step attempts to find a good assignment of the points chosen from the domain to a finite set of linear functions by solving LP problems, and the second step adds linear functions to the convex or concave component $g^+(x)$ and $g^-(x)$. We call this method \textit{KL23-LP}. A variant of this method is described in the same article, which we call \textit{KL23-LPrelaxed}. It handles some of the constraints of the LP model as lazy constraints, i.e., adding them to the LP model only when the previous solution does not satisfy them, since these LP problems require a large number of constraints. The numerical results on the set of instances from \citet{Rebennack15a} show that the KL23-LP and KL23-LPrelaxed methods can find feasible approximations to much larger instances than the KL21 method. The number of pieces in the constructed PWL functions is significantly smaller than for the RK2D method, but when KL21 finds a solution, it generally has fewer pieces than the solutions of the KL23-LP and KL23-LPrelaxed methods.

\citet{Duguet22a} proposes three principles based on which a heuristic producing high quality solutions for the \RCFP{} is designed: (i) General Piecewise Linear Representation: the algorithm constructs non necessarily continuous \PWL{} functions by partitioning the corridor domain into convex polytopes, each associated with a linear function that satisfies the corridor constraints on the polytope. This approach relaxes the continuity constraints imposed by KL21~/~KL23-LP~/~KL23-LPrelaxed, and generalizes beyond the simplex or rectangular partitions used in RK1D~/~RK2D. (ii) Greedy Polytope Selection: the algorithm constructs a linear piece at each iteration, defined over the largest possible polytope within the remaining domain, as determined by a specified metric. This subdomain is then excluded from further consideration. The process repeats until the entire corridor domain is covered. \citet{Duguet22a} show that a metric based on second partial derivatives results in fewer linear pieces than one based on domain area alone. (iii) Interior PWL Corridor Approximation: the corridor is approximated using an interior PWL corridor, allowing the construction of each linear piece to be formulated as a linear program. This is achieved by replacing the original corridor bounds $l$ and $u$ with two \PWL{} functions $\tilde{l}$ and $\tilde{u}$ that verify $l \leq \tilde{l} < \tilde{u} \leq  u$.

\subsubsection{The $\mathbb R^m$-corridor fitting problem}
All methods described for the \RCFP{} are adaptable to $\mathbb R^m$ for $m \geq 3$. To the best of our knowledge, the only known numerical results concern the particular case of the function $\prod_{i=1,...,m} x_i$ on domain $D=[0,1]^m$ for $m=3,4,5$ by methods \textit{KL23-LP} and \textit{KL23-LPrelaxed} in \citet{Kazda23}.

\section{Relaxations of the \RmCFP{}} \label{sec:relaxations}

The main difficulty in solving the \RmCFP{} lies in its infinite number of constraints. We propose four different relaxations involving only a finite number of constraints. They are based on a hypergraph coloring problem resulting from the discretization of the domain $D$ of $\mathit{Corridor}(u,l,D)$ and the identification of \textit{linearly compatible sets} defined hereafter.

\begin{definition}
(linearly compatible set). Given $\mathit{Corridor}(u,l,D)$, a finite set of points $s \subset D$ is \textit{linearly compatible} if there exists a linear function that is solution of the \RmCFP{} on $\mathit{Corridor}(u,l,\textit{conv}(s))$ where $\textit{conv}(s)$ is the convex hull of points in $s$. If there exists no such linear function, the set is called \textit{not linearly compatible}.
\end{definition}

A not linearly compatible set corresponds to a valid inequality for the CFP on $\mathit{Corridor}(u,l,D)$.
\begin{figure}
    \centering
    \includegraphics[width=0.5\linewidth]{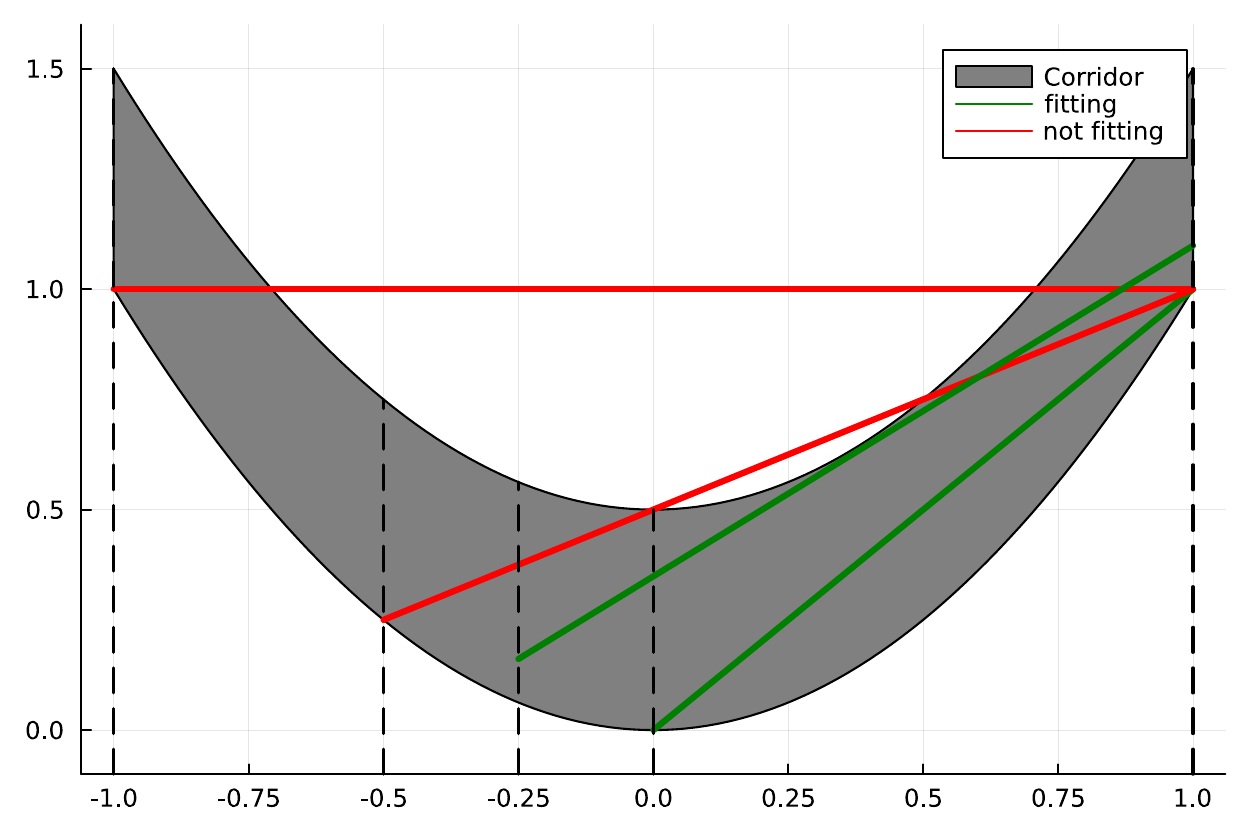}
    \caption{Examples of linear functions fitting and not fitting a corridor on $[-1,1]$ leading to linearly and not linearly compatible sets}
    \label{figure_example_linearly_compatible_set}
\end{figure}

\begin{example}
Consider the corridor $\mathcal{C}_0 = \mathit{Corridor}(x^2+\frac{1}{2},x^2,[-1,1])$ as well as the green and red linear functions fitting or not fitting $\mathcal{C}_0$ showed in Figure \ref{figure_example_linearly_compatible_set}.
The set $\{\{0\},\{1\}\}$ is linearly compatible because the linear function $x \mapsto x$ on $[0,1]$ is inside $\mathcal{C}_0$ on $[0,1]$, see the corresponding green line.
However, it is easy to see that there is no linear function solution to the \rCFP{} on corridor $\mathcal{C}_0$. Thus, in every solution of an \rCFP{} on a corridor that has a restriction to domain $[-1,1]$ equal to $\mathcal{C}_0$, the points $-1$ and $1$ are assigned to different linear pieces. Thus the set $\{\{-1\},\{1\}\}$ is a not linearly compatible set. Finally, for similar arguments, $\{\{-0.5\},\{1\}\}$ is a not linearly compatible set while $\{\{-0.25\},\{1\}\}$ is a linearly compatible set.
\end{example}

\subsection{Domain discretization, linearly compatible sets and resulting relaxation}
We introduce a relaxation of the \RmCFP{} based on the so-called \textit{domain discretization}.

\begin{definition}[Problem \relaxR{}]
Let corridor $\mathcal{C}$ be a corridor and $X\subset D$ be a finite set of points in the corridor domain. 
We denote \relaxR~the problem of partitioning $X$ into a minimum number of subsets subject to each subset being linearly compatible.
\end{definition}

\begin{proposition} \label{prop_relaxR}
    Consider a corridor $\C=\mathit{Corridor}(u,l,D)$ and a finite set $X \subset D$. The optimal solution of \relaxR~provides a lower bound to the \RmCFP{} on the corridor $\C$.
\end{proposition}

\begin{proof}
It suffices to show that to any optimal solution of the \RmCFP{}, there corresponds a feasible solution of \relaxR{}, and that the objective value of the two solutions are the same.
Let $g^{\mathrm{CFP}}$ be a piecewise linear function solution of the \RmCFP{}. We now build a feasible solution of problem \relaxR{} from $g^{\mathrm{CFP}}$. Associate to each point $x_i$ of $X$ a color $j$ corresponding to the piece $j$ of $g^{\mathrm{CFP}}$ whose domain contains $x_i$. 
The subsets of $X$ of one color are linearly compatible sets because they come from a feasible solution to the \RmCFP{}. Thus this association of colors to points of $X$ represents a feasible solution of \relaxR{}.
In addition, the number of colors used by this feasible solution is the same as the number of pieces of $g^{\mathrm{CFP}}$.
\end{proof}

\paragraph{Hypergraph coloring} 
Hypergraphs are a generalization of graphs where edges can connect more than two vertices. They should not be confused with multigraphs, which are graphs where several edges can connect the same pair of vertices. 
Let us consider a hypergraph $(V,E)$ and an integer $n>0$. An $n$-coloring of the vertices of $(V,E)$ is a function $c: V \mapsto \{1,...,n\}$ assigning a color to each vertex such that not all vertices of an edge $e \in E$ have the same color.
Figure \ref{figure_hypergraph_coloring} shows a 3-coloring for the hypergraph represented. Note that the three vertices connected by the yellow edge are not all of the same color.
\begin{figure}
    \centering
    \includegraphics[width=0.5\textwidth]{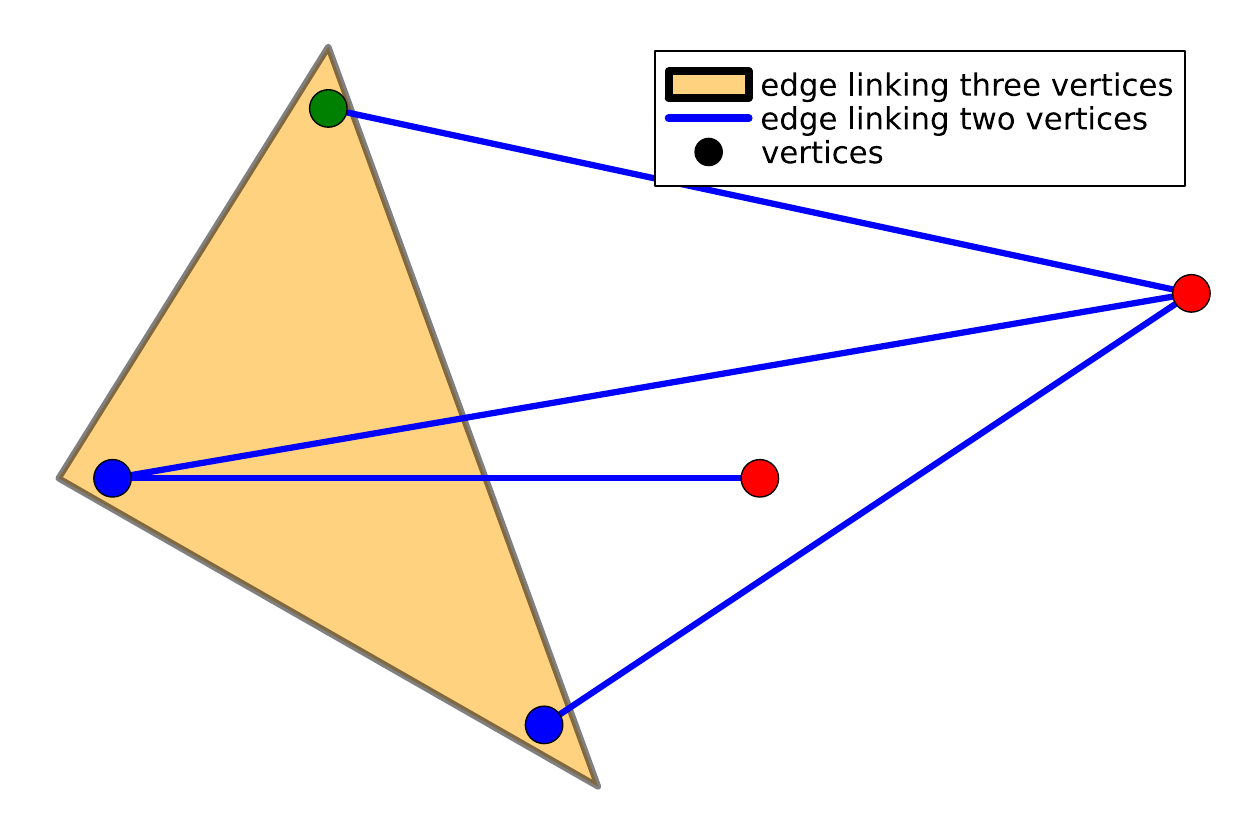}
    \caption{An example of hypergraph with a 3-coloring}
    \label{figure_hypergraph_coloring}
\end{figure}

\begin{definition}[Hypergraph coloring problem]
Let $H$ be a hypergraph and $n$ a positive integer. The problem of finding an $n$-coloring to hypergraph $H$ is called the \textit{hypergraph coloring problem}.
\end{definition}

More information on the hypergraph coloring problem can be found in \citet{Bretto13}.
Problem \relaxR{} on corridor $\mathcal{C}$ can be formulated and solved as a hypergraph coloring problem where hypergraph $H=(X,E)$ is defined as follows. The set of vertices $X$ is a finite set of points of the domain of $\mathcal{C}$.
There is one hyperedge $e \in E$ for each not linearly compatible subset of $X$.
We call such hypergraph a \textit{complete incompatibility hypergraph} of $X$ and each of its hyperedges an \textit{incompatibility edge} for $\mathcal{C}$ because it is built from a not linearly compatible set.

The number of edges of $E$ grows exponentially with $|X|$.
Moreover, determining whether a given subset of $X$ is linearly incompatible amounts to solving a decision variant of the \RmCFP{} to determine whether a feasible solution exists that uses only a single piece. Consequently, \relaxR{} is intractable except for very small sets $X$. The remaining of the section presents more tractable relaxations.

\subsection{\relaxRone~: hypergraph coloring of partial incompatibility hypergraphs}

We call \textit{incompatibility hypergraph} $H=(X,E)$, a hypergraph containing only edges that correspond to not linearly compatible sets. In particular, it can be obtained by removing some edges from a complete incompatibility hypergraph.
\begin{definition}\relaxRone{}~
Let $\mathcal{C}$ be a corridor, $X\subset D$ be a finite set of points in the corridor domain and $H=(X,E)$ be an incompatibility hypergraph for corridor $\mathcal{C}$. Problem \relaxRone{} is the hypergraph coloring problem applied to $H$.
\end{definition}
 
As \relaxRone{} is built from \relaxR{} by dropping some constraints, it is a relaxation of \relaxR{} and thus the following proposition holds:

\begin{proposition} \label{prop_relaxRone}
    Consider a corridor $\C=\mathit{Corridor}(u,l,D)$, a finite set $X \subset D$ and an incompatibility hypergraph $H=(X,E)$ for $\C$. An optimal value of \relaxRone{} provides a lower bound to the \RmCFP{} on the corridor $\C$.
\end{proposition}

Section \ref{sec:compute_edges} presents the different procedures to generate an incompatibility hypergraph $H$.

\subsection{\relaxRonestrong{} : a discrete variant of the \textit{Corridor Fitting Problem}}

We discuss another problem built from \relaxRone{} and additional constraints corresponding to valid inequalities of the \RmCFP{}.
\begin{definition} \relaxRonestrong{}
Given a corridor $\C=\mathit{Corridor}(u,l,D)$, a finite set $X \subset D$ and an incompatibility hypergraph $H=(X,E)$ for $\C$, we denote \relaxRonestrong{} the problem of finding a feasible solution of \relaxRone{} that minimizes the number of colors while also satisfying the three following constraints. 
\begin{enumerate}[label=(\roman*)]
\item\label{constraint_valid_inequalities} Each color $j$ is associated with a linear function 
    \begin{align}
        L_j(x) = a_j x + b_j, & \quad a_j \in \mathbb R^m, b_j \in \mathbb R \nonumber
    \end{align}
    defined on domain
        \begin{align}
        D_j=conv(\{x_i \in X | x_i \text{ is of color } j\} \nonumber
    \end{align}
\item\label{constraint_non_intersecting_domains} The interiors of the domains $D_j$ do not intersect
\item\label{constraint_def_discrete_pointwise} Each point $x_i \in X$ such that color $j$ is assigned to $x_i$ verifies the corridor constraint:
    \begin{align}
        l(x_i) \leq L_j(x_i) \leq u(x_i). \nonumber
    \end{align}
\end{enumerate}
\end{definition}
\relaxRonestrong{} can be seen as a discrete version of the \textit{Corridor Fitting Problem}. 
A feasible solution of \relaxRonestrong{} is ``closer to a \PWL{} function fitting the corridor $\C$'' than a feasible solution of \relaxRone{} because it defines $n$ linear functions $L_j$ (Constraint \ref{constraint_valid_inequalities}), the interior of the domains of those linear functions do not intersect (Constraint \ref{constraint_non_intersecting_domains}), and $L_j$ verifies the corridor constraints on $X$ (Constraint \ref{constraint_def_discrete_pointwise}).

\begin{proposition} \label{prop_relaxation}
    Consider a corridor $\C=\mathit{Corridor}(u,l,D)$, a finite set $X \subset D$ and an incompatibility hypergraph $H=(X,E)$ for $\C$. The optimal solution of \relaxRonestrong{} provides a lower bound to the \RmCFP{} on the corridor $\C$.
\end{proposition}
\begin{proof}
As in the proof of Proposition \ref{prop_relaxR}, it suffices to show that there is a correspondence between any optimal solution of the \RmCFP{} and a feasible solution of \relaxRonestrong{}, and that the objective value of the two solutions are the same.
Let $g^{\mathrm{CFP}}$ be a piecewise linear function solution of the \RmCFP{}. We now build a feasible solution of the model \relaxRonestrong{} from $g^{\mathrm{CFP}}$. Associate to each point $x_i$ of $X$ a color $j$ corresponding to the piece $j$ of $g^{\mathrm{CFP}}$ whose domain contains $x_i$. In addition, define the linear functions $L_j$ with the linear function $j$ of $g^{\mathrm{CFP}}$ on a smaller domain that is the convex hull of points of $X$ colored in color $j$. Remark that all constraints of \relaxRonestrong{} are satisfied by construction. In addition, the objective value of $g^{\mathrm{CFP}}$ and the feasible solution of \relaxRonestrong{} constructed are the same.
\end{proof}

\begin{example}
Figure \ref{figure_DCFP_hypergraph_example} represents the domain of a corridor, with a hypergraph $H$ whose vertices are the black points and whose edges are in blue and yellow. These two edges are incompatibility edges for the instance of the \RmCFP{} for which they were computed. More precisely, for any feasible solution of this instance of the \RmCFP{}, each of these edges must be covered by at least two different linear pieces. Figure \ref{figure_superposition_realisable_aretes} shows a feasible solution of this instance in green, on which the edges of $H$ are superimposed. Remark that both edges are covered by at least two different pieces as intended.
\end{example}
\begin{figure}
    \centering
    \begin{subfigure}[b]{.47\textwidth}
        \includegraphics[width=\linewidth]{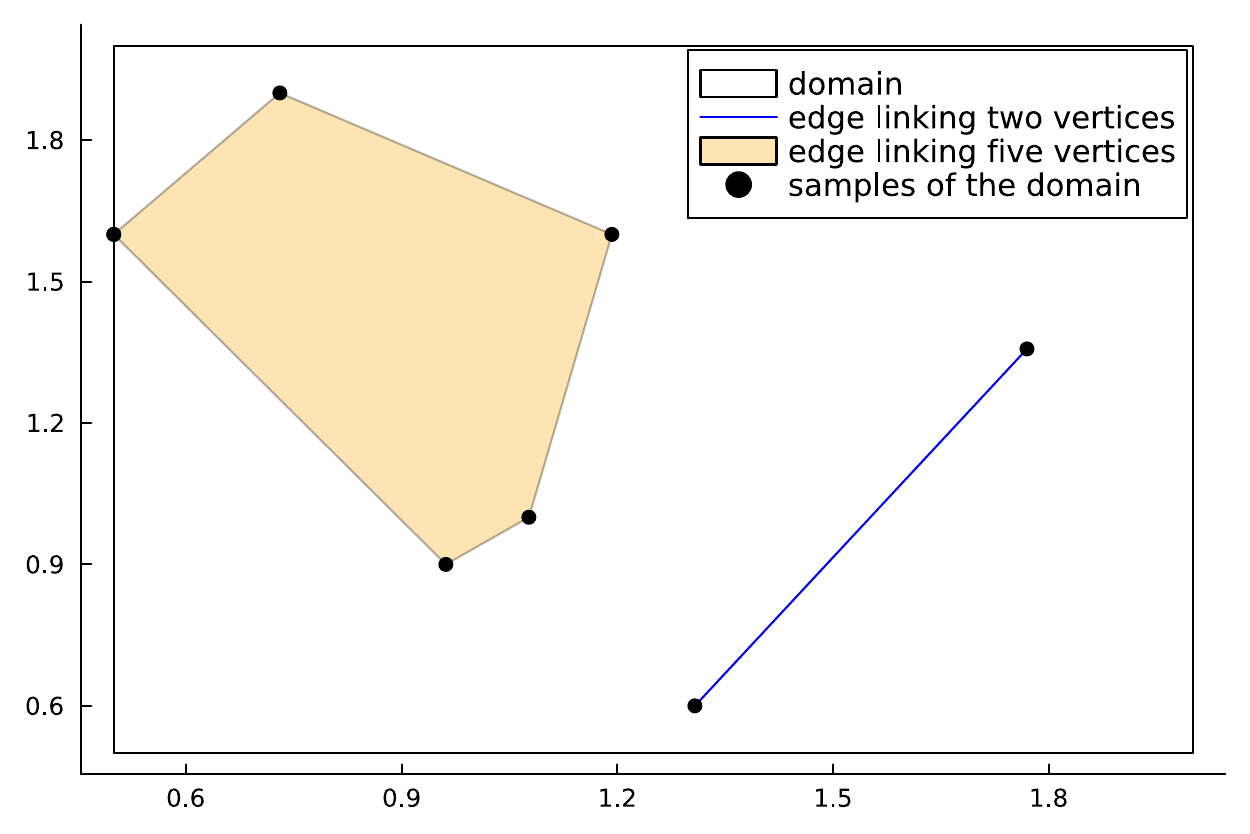}

        \caption{Representation of a hypergraph $H$ with two incompatibility edges}
        \label{figure_DCFP_hypergraph_example}
    \end{subfigure}
    \begin{subfigure}[b]{0.47\textwidth}
        \includegraphics[width=\linewidth]{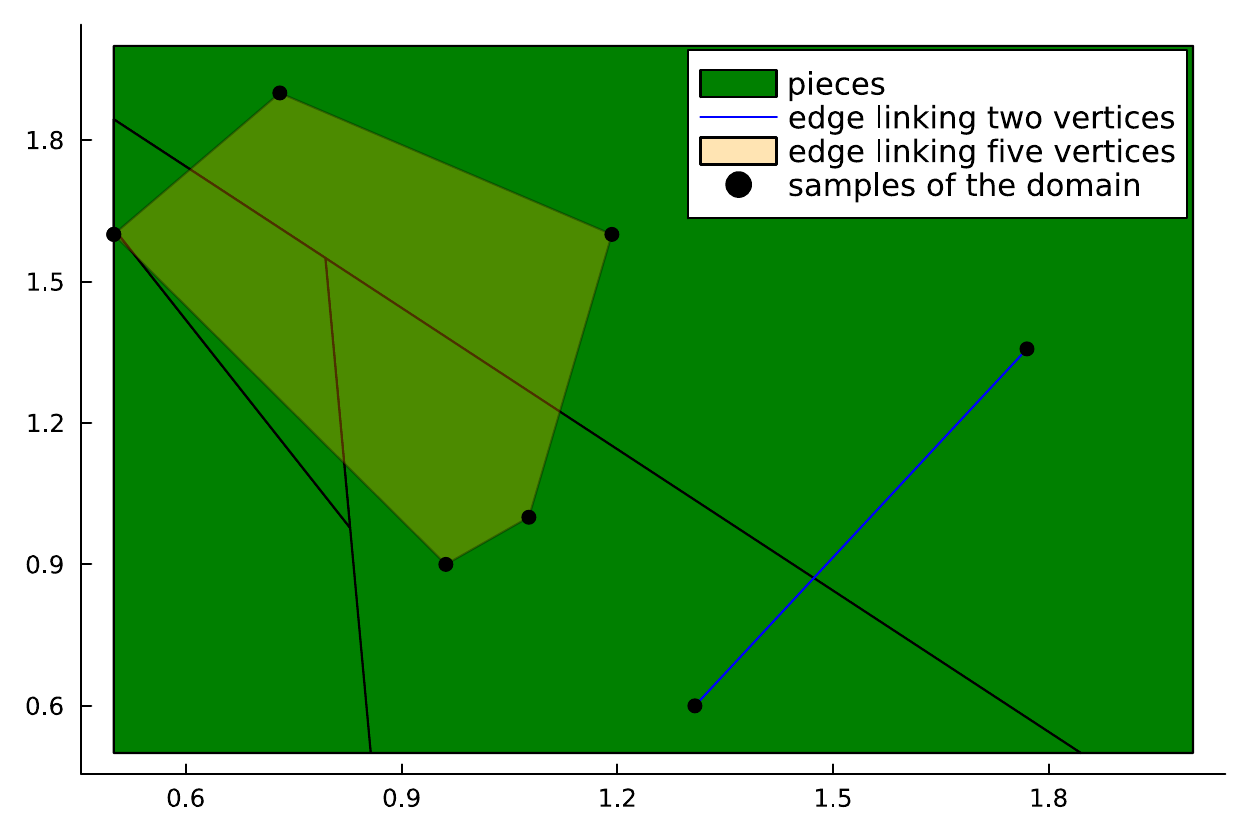}

        \caption{Superposition of incompatibility edges and a feasible solution of the \RmCFP{}}
        \label{figure_superposition_realisable_aretes}
    \end{subfigure}
    \caption{Representation of the hypergraph $H$ in an instance of \relaxRonestrong{}}
    \label{figure_example_DCFP}
\end{figure}

\subsection{\relmaximum{} and \relmaximal{}: clique-based relaxations}

We consider the two following clique-based problems. 

\begin{definition}\relmaximum{}~
Let $\mathcal{C}$ be a corridor, $X\subset D$ a finite set of points in the corridor domain, and $H=(X,E)$ an incompatibility hypergraph for corridor $\mathcal{C}$. Let $G$ be the incompatibility graph obtained from $H$ by removing every edge connected to at least three vertices. Problem \relmaximum{} consists in finding a clique with maximum cardinality among the cliques of $G$. 
\end{definition}

\begin{definition}\relmaximal{}~
Let $\mathcal{C}$ be a corridor, $X\subset D$  a finite set of points in the corridor domain, and $H=(X,E)$  an incompatibility hypergraph for corridor $\mathcal{C}$. Let $G$ be the incompatibility graph obtained from $H$ by removing every edge connected to at least three vertices. Problem \relmaximal{} consists in finding a clique that is not a subset of another clique of $G$. 
\end{definition}
\begin{proposition}
    Consider a corridor $\C=\mathit{Corridor}(u,l,D)$, a finite set $X \subset D$, an incompatibility hypergraph $H=(X,E)$ for $\C$ and a graph $G$ obtained by dropping all edges with at least three vertices in $H$. The optimal solutions of \relmaximum{} and \relmaximal{} provide lower bounds to the \RmCFP{} on the corridor $\C$.
\end{proposition}
\begin{proof}
The cardinal of a maximal clique is a lower bound to the cardinal of a maximum clique, which is itself a lower bound to the chromatic number of a hypergraph. Thus, both the optimal values of \relmaximum{} and \relmaximal{} are lower bounds of \relaxRone{} which is a lower bound to the \RmCFP{} according to Proposition \ref{prop_relaxRone}.
\end{proof}

\subsection{Relative Strengths of the Relaxations}

The following proposition states the relative strengths of the four relaxations defined previously as well as the \RmCFP{}. The proof can be deduced from the derivation of the relaxations.

\begin{proposition}
    Let a corridor $\C=\mathit{Corridor}(u,l,D)$, a finite set $X \subset D$ and an incompatibility hypergraph $H=(X,E)$ for $\C$. Then the optimal values of the four relaxations and the \RmCFP{} satisfy:
    $$ z_{CFP} \geq z_{DCFP} \geq z_{color} \geq z_{MumC} \geq z_{MalC\,.} $$
\end{proposition}

Obviously, the strength of a relaxation is decreasing with the amount of information of the original problem that is discarded. \relaxRonestrong{} uses the hypergraph coloring structure as well as other specific information such as satisfying the corridor constraints on the vertices and convexity information. \relaxRone{} only uses the hypergraph coloring structure. The two clique relaxations only use the graph coloring structure. We will see in the numerical experiments that the quality of the lower bounds produced in a given time follows the same trend.

The following result gives the complexity of the relaxations.

\begin{proposition} \label{prop_complexity_relaxations}
    \relaxRonestrong{}, \relaxRone{} and \relmaximum{} are NP-hard whereas \relmaximal{} is in P.
\end{proposition}

The proof can be found in Appendix \ref{section_complexity_proof}.

\section{Computation of incompatibility hypergraphs and resulting lower bounding algorithms} 
\label{sec:compute_edges}

The implementation of algorithms leveraging the relaxations of Section \ref{sec:relaxations} have two necessary components: producing an incompatibility hypergraph and solving the corresponding relaxations. We explain in this section how the first component is implemented and how the second component is included in algorithms.
In practice, the computation of incompatibility edges must be fast, since a numerical algorithm often requires the identification of many such edges. Therefore we propose heuristic procedures rather than exact ones to test whether a finite subset of domain points constitutes an incompatibility edge. 
Section \ref{subsec:construction_vertices} describes how to select the set of vertices of incompatibility hypergraphs, while Sections \ref{subsec:construction_aretes_2} and \ref{subsec:construction_aretes_multi} detail methods to compute incompatibility edges. 
Finally, Section \ref{subsec:algo_relaxations} gives the main idea behind our implementations of algorithm leveraging the relaxations of Section \ref{sec:relaxations} to produce lower bounds to the \RmCFP{}. Details are given in Appendix \ref{sec:detail_algorithms}. The pseudo-code is generic but our implementation for the computational experimentation is specific to the two-variable case.

\subsection{Computation of the vertices} \label{subsec:construction_vertices}
Consider a $\mathit{Corridor}(u,l,D)$ and a number $p_{\textit{vertices}}$ of vertices wanted.
The idea of our implementation is one of the many ways to get an ``approximate uniform sampling'' while producing exactly $p_{vertices}$ vertices for comparison purposes.
If the domain is not a hyperrectangle, find a hyperrectangle containing the domain. Then, find a sampling size so that a regular grid with extreme points the vertices of the hyperrectangle would possess close but less points than the number of points wanted inside the domain.
Sample the hyperrectangle according to the regular grid and remove samples that do not belong to the domain. Finally, produce enough new samples to reach $p_{vertices}$ points, possibly by adjusting a few samples to keep the pseudo ``regular'' property of the sampling. 
%

\subsection{Computation of incompatibility edges of two Vertices} \label{subsec:construction_aretes_2}

Consider a $\mathit{Corridor}(u,l,D)$ and let $x_1$ and $x_2$ be two points in the domain $D$. The edge $e=(x_1,x_2)$ is an incompatibility edge if and only if the \RmCFP{} on $\mathit{Corridor}(u, l, conv(x_1,x_2))$ does not admit a solution with one piece. This problem can be seen as an \rCFP{} because a function $f$ of several variables can be restricted to a one-variable function $t$ on a line segment $s = conv(x_1,x_2)$. Indeed, using a convex combination we have $\restriction{f}{s}(t) = f(x(t)) = f(t x_1 + (1-t) x_2)$.

The heuristic method we propose consists in checking if there exists parameters $a$ and $b$ such that $at+b$ is a linear function satisfying the corridor on a finite subset of points $T$ selected from the line segment $s = conv(x_1,x_2)$. This problem can be formulated as a linear feasibility problem with two variables $a$ and $b$ and $2|T|$ constraints:
\begin{equation}
    (\mathbb R\textrm{-LDRFP})\quad l(t_i) \leq a t_i+b \leq u(t_i) \quad \forall t_i \in T, \label{eq_R-DRFP}
\end{equation}
where LDRFP stands for Linear Data Range Fitting Problem.
If there are no real numbers $a$ and $b$ solution of $\mathbb R\textrm{-LDRFP}$, then the edge $\{x_1,x_2\}$ is proven to be an incompatibility edge; however, if there exists such real numbers, it does not prove that it is not an incompatibility edge because we solved a relaxation of the \rCFP{} on $\mathit{Corridor}(u, l, conv(x_1,x_2))$. 
\begin{algorithm}[!ht]
\textbf{Input:} X a finite set of vertices \\
\textbf{Output:} $E_2$ a set of incompatibility edges of size two
\begin{algorithmic}[1]
\State $E_{2} = \{\}$ 
\For{$x_1,x_2$ in $X$} 
    \State $e = (x_1,x_2)$
    \If{\textsc{IsAnIncompatibilityEdge}(e)} \Comment{if $\mathbb R\textrm{-LDRFP}$ \eqref{eq_R-DRFP} is infeasible}
        \State $E_{2} \xleftarrow{} E_{2} \cup e$
    \EndIf
\EndFor
\end{algorithmic}
\caption{Compute all incompatibility edges of two vertices}
\label{algorithm_build_edge_sizetwo}
\end{algorithm}

In our numerical implementation, $T$ is a regular sampling of the line segment $s$ with $N_{samples}$ points. To solve the $\mathbb{R}$-LDRFP, we use the algorithm of \citet{O'Rourke81} modified to stop as soon as it is shown that no line segment can cover all the interval. It has linear complexity in $|T|$ and is therefore very fast.

\subsection{Computation of incompatibility edges of at least three vertices} \label{subsec:construction_aretes_multi}

Consider the corridor $\mathcal{C}=\mathit{Corridor}(u,l,D)$ and a finite set $X \subset D$. 
A subset of points $e \subset X$ is an incompatibility edge for the corridor $\mathcal{C}$ if there is no linear function that is a solution to the \RmCFP{} on $\mathit{Corridor}(u, l, conv(e))$.
Thus, the problem of checking if an edge is an incompatibility edge has a natural model as a semi-infinite program. Here is such a mathematical model: find $a \in \mathbb R^m$ and $b \in \mathbb R$ such that
\begin{equation}
    l(x) \leq a^Tx+b \leq u(x) \quad \forall x \in conv(e),
\end{equation}
where there is an infinite number of constraints due to the infinite number of points of the convex hull. For more detail on semi-infinite programming, we refer the reader to \cite{Djelassi21_SIPreview}. 
Semi-infinite programming problems are hard to solve in general, so we propose a heuristic method: solving a discretization-based relaxation of the problem. Given a finite set $S \subset conv(e)$, the problem consists in finding $a \in \mathbb R^m$ and $b \in \mathbb R$ such that:
\begin{equation}
    (\mathbb R^m\textrm{-LDRFP}) \quad l(x_i) \leq a^Tx_i+b \leq u(x_i) \quad \forall x_i \in S \subset conv(e). \label{eq_R2-DRFP}
\end{equation}
It can be formulated as a linear program with $m+1$ variables and $2|S|$ constraints.
Any instance for which there is no solution to the relaxation implies $e$ is an incompatibility edge for $\mathcal{C}$.
Thus, at the cost of possibly failing to detect some incompatibility edges we enforce the corridor constraints only on a finite subset $S$ of points of $conv(e)$.
\begin{algorithm}[!ht]
\textbf{Input:} \\
- X a finite set of vertices, \\
- $E_2$ a set of incompatibility edges of size two, \\
- $k_{repetition}$ the number of repetitions for each vertex \\
\textbf{Output:} $E_{3+}$ a set of incompatibility edges with three or more vertices
\begin{algorithmic}[1]
\State $E_{3+} = \{\}$ 
\For{$x$ in $X$} \label{algo_outer_for_Em} 
    \For{$k_{rep}$ in $\{1, ..., k_{repetition}\}$}
        \State $e \xleftarrow{} \{x\}$ \Comment{$e$ is the edge under construction} \label{algo_init_arete}
        \For{$y$ in $\textsc{Shuffle}(X \setminus e)$}
            \If{$(y,s) \notin E_2 \quad \forall s \in e$} \label{algorithm_line_arete}
                \If{$y \notin conv(e)$}    \Comment{removes non extremal points}\label{algo_test_inside} 
                    \State $e \xleftarrow{} e \cup \{y\}$
                    \If{\textsc{IsAnIncompatibilityEdge}(e)} \Comment{if $\mathbb R^m\textrm{-LDRFP}$ \eqref{eq_R2-DRFP} is infeasible}\label{algo_line_test_incompatibility}
                        \State $E_{3^+} \xleftarrow{} E_{3^+} \cup e$
                        \State \textbf{break}
                    \EndIf
                \EndIf
            \EndIf
        \EndFor
    \EndFor
\EndFor
\end{algorithmic}
\caption{Compute incompatibility edges of at least three vertices}
\label{algorithm_build_edge_many}
\end{algorithm}

Algorithm \ref{algorithm_build_edge_many} describes how we compute incompatibility edges of three or more vertices. Due to the exponential number of potential incompatibility edges, it does not check all potential edges:
it only attempts to build $|X| \times k_{repetition}$ edges.
Starting from a subset containing only the vertex $x$, the algorithm adds one by one vertices $y$ randomly chosen among the remaining vertices until either the subset is proved to be an incompatibility edge or all vertices have been added.

If $e \subset e'$ is an incompatibility edge, then $\mathit{Corridor}(u,l,conv(e'))$ is also an incompatibility edge. In this case adding $e$ and $e'$ to the hypergraph is redundant and we keep only the tightest constraint among the two, i.e., the one induced by $e$.
That is why the test on Line \ref{algorithm_line_arete} ensures that the new edge does not contain a size two incompatibility edge, and Line \ref{algo_test_inside} prevents vertices $y \in conv(e)$ from being added. Finally, the test on Line \ref{algo_line_test_incompatibility} solves the $\mathbb R^m\textrm{-LDRFP}$ on the corridor given in input with finite subset of points $S$. It returns false if there exists a solution and true if not. The implementation details on the construction of the set $S$ are discussed in Appendix \ref{sec:construction_S}.

\subsection{Lower bounding algorithms leveraging the relaxations} \label{subsec:algo_relaxations}

We derive one algorithm for each of the four relaxations from Section \ref{sec:relaxations}. Each algorithm \algcolor{}, \algmaximum{} and \algmaximal{} iteratively computes lower bounds for the \RmCFP{} by solving its corresponding relaxation on incompatibility hypergraphs (or graphs) derived from progressively finer discretizations of the domain $X \subset D$, until the time limit is reached. Algorithm \algmilp{} computes lower bounds by solving the decision version of \relaxRonestrong{} for $n$ pieces and hypergraph $H$, denoted \relaxRonestrongdecision{}.

Given the structural similarities among these methods, we provide a generic pseudo-code hereafter and defer the detailed descriptions to Appendix \ref{sec:detail_algorithms}.

\begin{algorithm}[!ht]
\begin{algorithmic}[1]
\While{time remaining $>0$} 
    \State Build the incompatibility graph or hypergraph $H=(X,E)$.
    \State Solve the relaxation \relaxRonestrongdecision{}, \relaxRone{}, \relmaximum{} or \relmaximal{}.
    \State Update the best known lower bound $\mathit{LB}$ or increase the discretization size $|X|$
\EndWhile
\State $\text{Return }$the best known lower bound $\mathit{LB}$
\end{algorithmic}
\caption{Lower bounding algorithms \algmilp{}, \algcolor{}, \algmaximum{} and \algmaximal{}}
\label{algorithm_lowerbounds}
\end{algorithm}

\section{Improved Upper Bounds of the \RCFP{}} 
\label{sec:new_UB}
We describe the improvement of two heuristics from the literature.

\subsection{DLN: improvement of Algorithm 1 of \citet{Duguet22a}} \label{subsection_DLN_improvement}
We improve the implementation of the algorithm 1 of \citet{Duguet22a} used to compute feasible solutions to the \RCFP{} and recalled in Section \ref{subsubsec:stateofartR2cfp} of the literature review of this paper.

Specifically, the implementation of the computation of a linear piece fitting a corridor on a convex subdomain that is 'as large as possible', referred to as \textit{maximal piece in direction $d$ problem} \citep{Duguet22a}, has been improved. It results in a linear piece covering a larger domain with the drawback of being more time-consuming. Appendix \ref{section_improvement_DLN} provides more details on this improvement as well as another implementation improvement due to a change of library for the interval analysis operations.

The implementation has been made publicly available\footnotemark{}\footnotetext{https://github.com/LICO-labs/two-variable-function-PWL-approximation} and the method is refered to as DLN in the numerical results presented in Section \ref{subsection_results_heuristics}.

\subsection{\iLinA{}: improvement of LinA2D} \label{subsection_iLinA2D_improvement}

In the literature, heuristics RK1D \citep{Rebennack15a} and LinA2D \citep{Duguet22a} approximate a bivariate function $f(x,y)$ with a target precision $\delta$ by applying a transformation $\widetilde{\gamma}(x)$ and decomposing $f$ into a combination of two univariate functions $\widetilde{f}_1(x)$ and $\widetilde{f}_2(x)$, which are then approximated separately to obtain rectangular pieces. LinA2D produces approximations with fewer pieces than RK1D because it is not constrained to continuous piecewise‑linear representations: discontinuity is allowed. Both heuristics must split the error budget $\widetilde{\gamma}(\delta)$ between the two univariate approximations. In practice, each method simply divides the budget equally between the two terms. However, any split of the budget would lead to a valid approximation.

To address this limitation for linearly separable functions $f(x,y)=f_1(x)+f_2(y)$ we introduce \iLinA{} which optimally splits the error budget between the two univariate approximations.  It is summarized in Algorithm \ref{algorithm_iLinA2D}. Starting from the feasible solution provided by an equal split of the error budget, \iLinA{} deduces the interval $[n_1^{\textrm{min}}, n_1^{\textrm{max}}]$ that contains the optimal number of linear pieces required for the piecewise linear approximation of $f_1$ in any valid error split. Then for each integer $n_1$ in $[n_1^{\textrm{min}}, n_1^{\textrm{max}}]$, the algorithm determines the minimum approximation error $\gamma_1$ that yields a piecewise linear approximation of $f_1$ with exactly $n_1$ pieces. The remaining error budget, $\gamma_2 = \delta - \gamma_1$, is then allocated to the approximation of  $f_2$. The corresponding number of pieces, $n_2$, required to approximate $f_2$ with approximation error $\gamma_2$, is subsequently computed. If the total number of pieces $n_1n_2$ improves upon the best known solution $z^{\textrm{best}}$, the resulting configuration is retained as the new incumbent.

\begin{algorithm}[!ht]
\textbf{Input:} \\
- $\delta$ the target precision or total error budget \\
- $f_1, f_2$ the two one-variable functions that verify $f(x,y)=f_1(x)+f_2(y)$\\
\textbf{Output:} $n_1^{\textrm{best}}, n_2^{\textrm{best}}, z^{\textrm{best}}$ the best solution after error split optimization

Compute a feasible solution with an equally split error budget

\begin{algorithmic}[1]
\State $n_1 \xleftarrow{}$ \textsc{MinNumberOfPiecesUnivariate}(${f}_1$, $\frac{\delta}{2}$) \Comment{see Algorithm \ref{algorithm_min_nber_pieces}} 
\State $n_2 \xleftarrow{}$ \textsc{MinNumberOfPiecesUnivariate}(${f}_2$, $\frac{\delta}{2}$) 
\State $\{n_1^{\textrm{best}}, n_2^{\textrm{best}}, z^{\textrm{best}}\} = \{n_1, n_2, n_1 n_2\}$ 
\end{algorithmic}
Compute a valid range of values for $n_1$

\begin{algorithmic}[1]
\setcounter{ALG@line}{3}
\State $n_1^{\textrm{min}} \xleftarrow{}$ \textsc{MinNumberOfPiecesUnivariate}(${f}_1$, $\delta$) 
\State $n_2^{\textrm{min}} \xleftarrow{}$ \textsc{MinNumberOfPiecesUnivariate}(${f}_2$, $\delta$) 
\State $n_1^{\textrm{max}} = \left\lfloor \frac{ z^{\textrm{best}} }{n_2^{\textrm{min}}} \right\rfloor $
\end{algorithmic}

Explore the range of values for $n_1$ to find the optimal split of the error budget

\begin{algorithmic}[1]
\setcounter{ALG@line}{6}
\For{$n_1$ in $[n_1^{\textrm{min}}, n_1^{\textrm{max}}]$} 
    \State $\gamma_1 \xleftarrow{}$ \textsc{MinErrorUnivariate}($f_1$, $n_1$) \Comment{see Algorithm \ref{algorithm_iLinA2D_detailed}}
    \State $\gamma_2 = \delta - \gamma_1$
    \State $n_2 \xleftarrow{}$ \textsc{MinNumberOfPiecesUnivariate}($f_2$, $\gamma_2$) 
    \If{$n_1n_2 < z^{\textrm{best}}$} 
        $\{n_1^{\textrm{best}}, n_2^{\textrm{best}}, z^{\textrm{best}}\} = \{n_1, n_2, n_1 n_2\}$ 
    \EndIf
\EndFor
\end{algorithmic}

Return the best solution found

\begin{algorithmic}[1]
\setcounter{ALG@line}{13}
\State $\text{Return } n_1^{\textrm{best}}, n_2^{\textrm{best}}, z^{\textrm{best}}$ 
\end{algorithmic}
\caption{Algorithm \iLinA{}}
\label{algorithm_iLinA2D}
\end{algorithm}

By design \iLinA{} outperforms LinA2D in piece count: it matches its performance when it is optimal to evenly divide the error budget, and produces strictly better results when an uneven split reduces the number of pieces of the solution. It employs two main procedures, \textsc{MinNumberOfPiecesUnivariate} and \textsc{MinErrorUnivariate}, both efficiently implemented using PiecewiseLinApprox.jl\footnote{https://github.com/LICO-labs/PiecewiseLinApprox.jl} (formerly LinA.jl) as described in Appendix \ref{section_appendix_findminerror_minnberpieces}. 

\section{Numerical Results} \label{sec:experiences_numeriques}

We present and analyze the numerical results, on the \RCFP{}, of the four lower bounding algorithms described in Section \ref{subsec:algo_relaxations} and the best-known heuristics. Nevertheless, the theoretical foundation of the lower bounding algorithms is valid for the generic \RmCFP{}.
The section is organized as follows. The instance set and parameter settings are described in Section \ref{subsection_instances_params}. The four lower bounding algorithms described in Section \ref{subsec:algo_relaxations} are compared in Section \ref{subsection_results_lowerbounds}, together with two variants of \algmilp{} designed to show the influence of the presence of hypergraph coloring constraints on the MILP model \eqref{model_DCFP}. The upper bounds produced by algorithms of the literature are summarized in Section \ref{subsection_results_heuristics}. Finally, we compare the best lower bounds with the best upper bounds of the \RCFP{} in Section \ref{subsection_optimal_solutions} to prove for the first time the optimality of a significant proportion of the instance set.

\subsection{Instance Set and Parameter Settings} 
\label{subsection_instances_params}

We use the standard instance set of the literature on the \RCFP{} \citep{Rebennack15a}. It consists in 45 instances created from 9 nonlinear bivariate functions and 5 different values of the absolute approximation error $\delta$. 
The two first functions are linearly separable while the seven others are non-linearly separable, so we call them $L_1$, $L_2$, $N_1$, ..., $N_7$. Table~\ref{heur_table_functions_definition} gives the expression and the domain of each function.  Their graphs are showed in Figure \ref{heur_figure_benchmark}.

\begin{table}[!ht]
    \centering
    \caption{Expressions and domains of the functions in the instance set}
    \begin{tabular}{ccc} 
        \toprule
        ref & expression & domain \\ \midrule 
        $L_1$ & $x^2-y^2$ & $[0.5,7.5] \times [0.5,3.5]$ \\ 
        $L_2$ & $x^2+y^2$ & $[0.5,7.5] \times [0.5,3.5]$ \\ 
        $N_1$ & $x y$ & $[2,8] \times [2,4]$ \\ 
        $N_2$ & $x \exp^{-x^2-y^2}$ & $[0.5,2] \times [0.5,2$] \\ 
        $N_3$ & $x \sin(y)$ & $[1,4] \times [0.05,3.1]$ \\ 
        $N_4$ & $\frac{\sin(x)}{x}y^2$ & $[1,3] \times [1,2]$ \\ 
        $N_5$ & $x \sin(x) \sin(y)$ & $[0.05,3.1] \times [0.05,3.1]$ \\ 
        $N_6$ & $(x^2-y^2)^2$ & $[1,2] \times [1,2]$ \\ 
        $N_7$ & $\exp^{-10 (x^2-y^2)^2}$ & $[1,2] \times [1,2]$ \\ \bottomrule
    \end{tabular}
    \label{heur_table_functions_definition}
\end{table}

\begin{figure}[!ht]
\centering

\begin{subfigure}[b]{.32\linewidth}
\includegraphics[width=\linewidth]{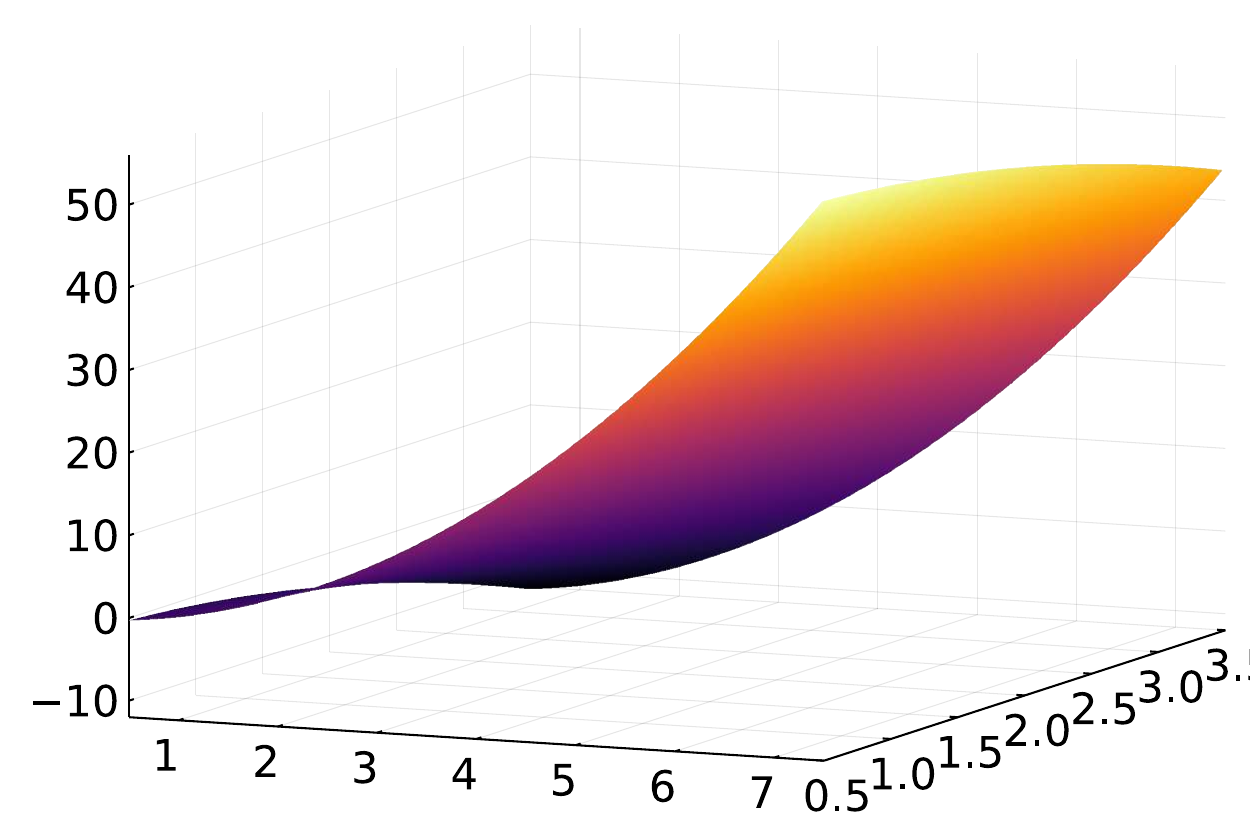}
\caption{$L_1$}\label{heur_figure_benchmark_L1}
\end{subfigure}
\begin{subfigure}[b]{.32\linewidth}
\includegraphics[width=\linewidth]{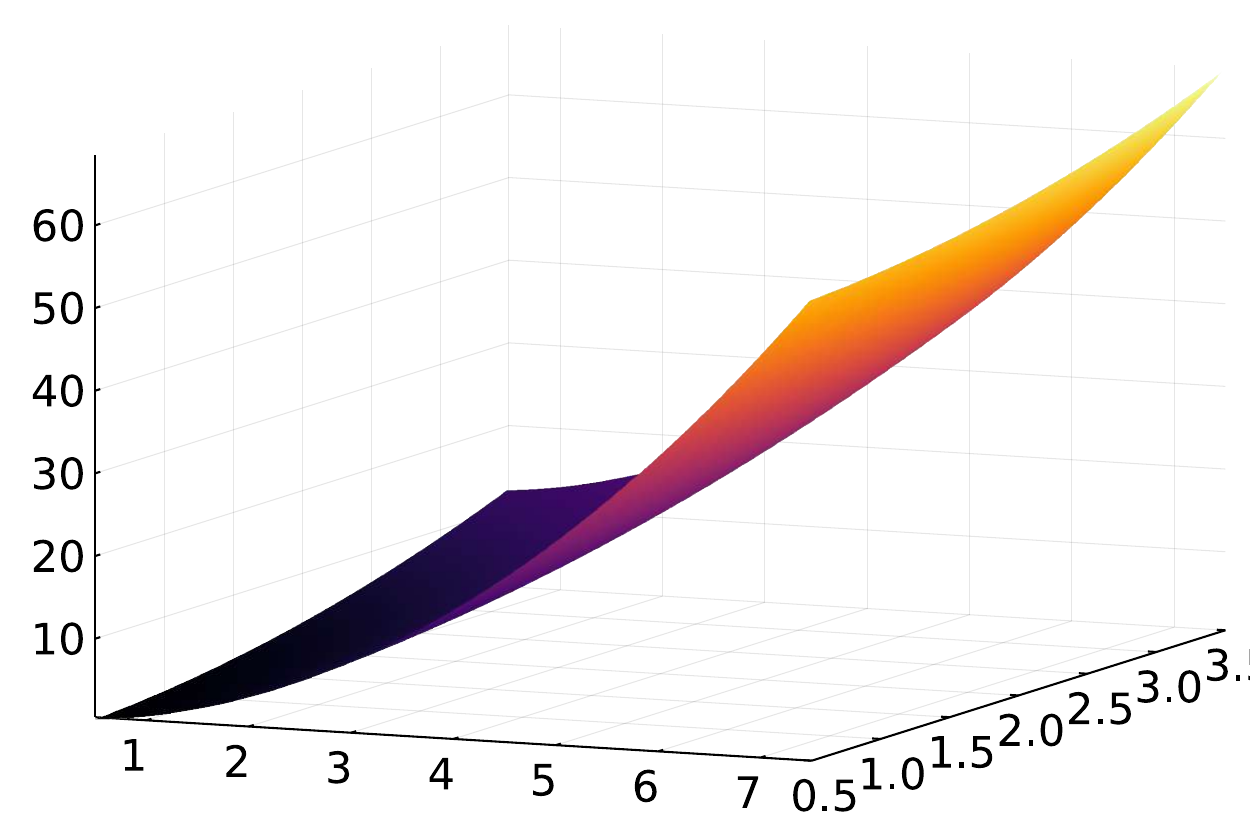}
\caption{$L_2$}\label{heur_figure_benchmark_L2}
\end{subfigure}
\begin{subfigure}[b]{.32\linewidth}
\includegraphics[width=\linewidth]{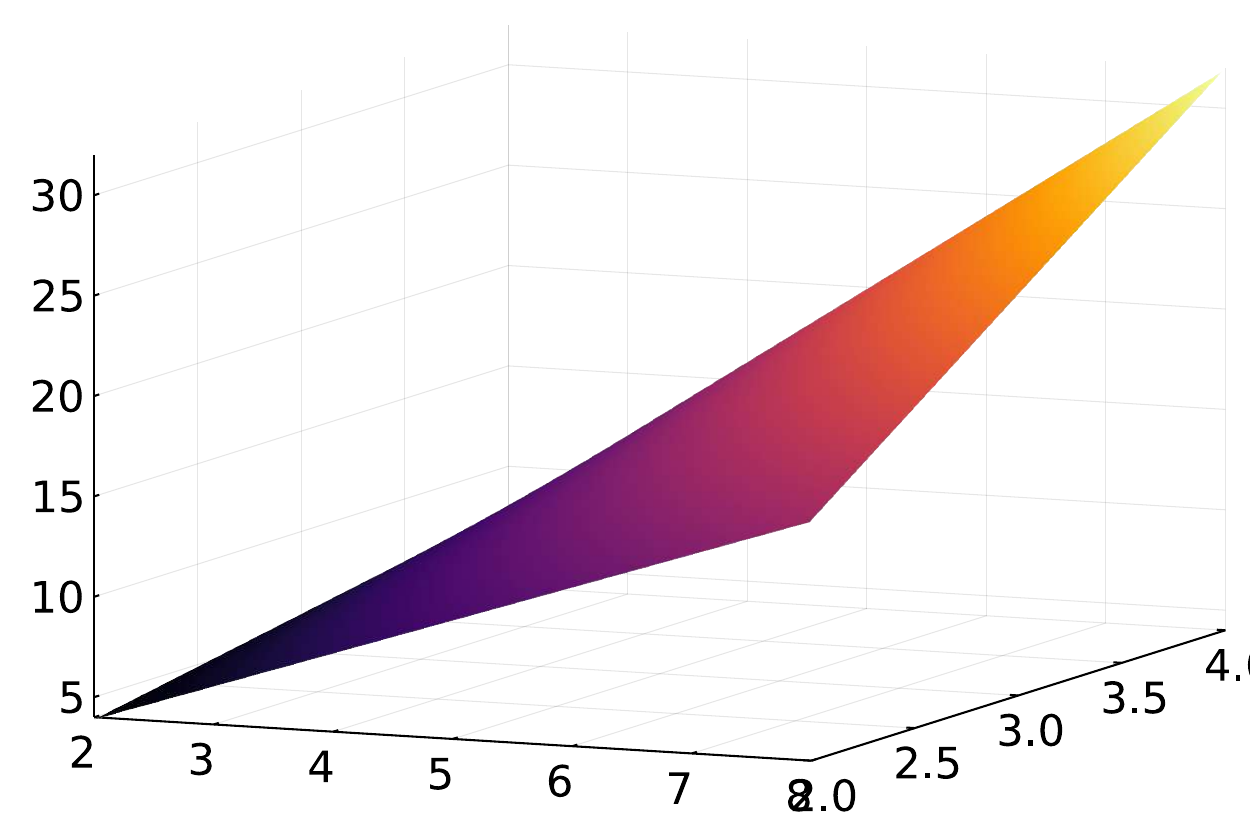}
\caption{$N_1$}\label{heur_figure_benchmark_N1}
\end{subfigure}

\begin{subfigure}[b]{.32\linewidth}
\includegraphics[width=\linewidth]{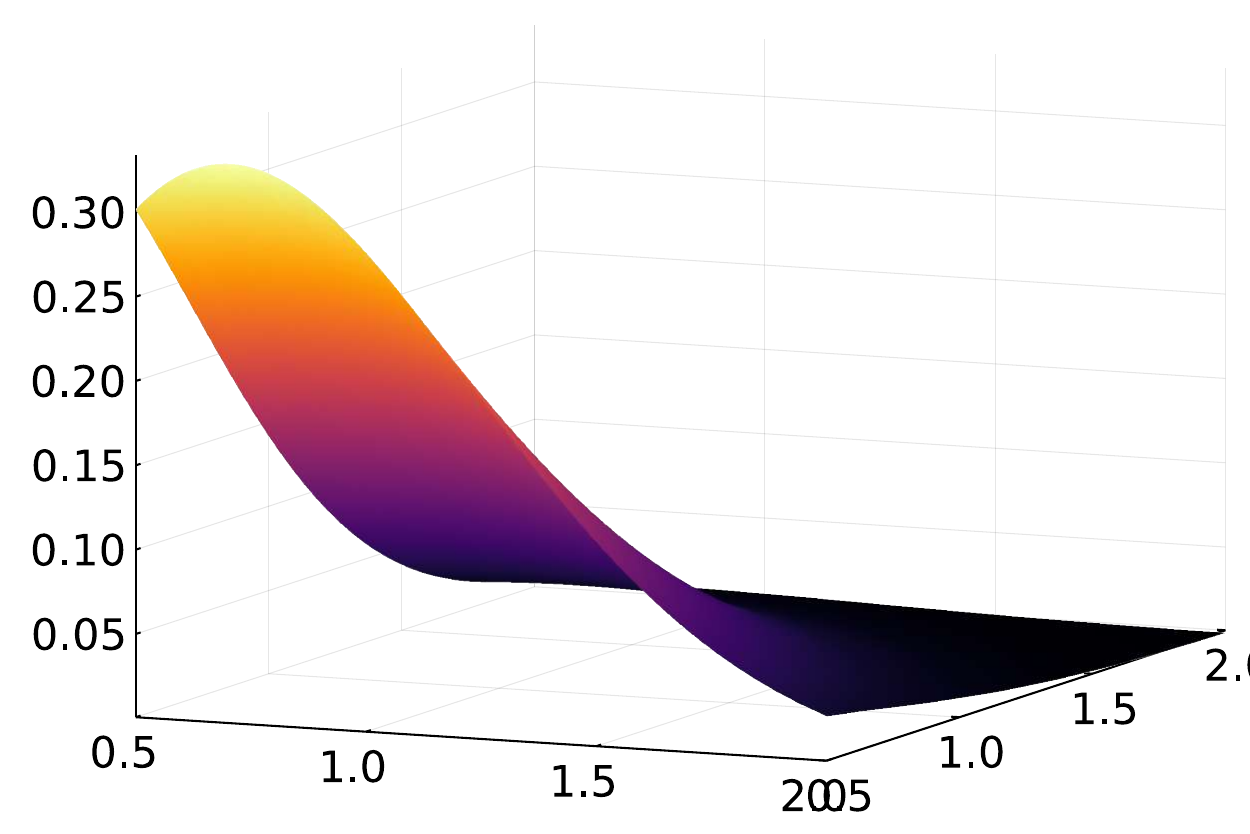}
\caption{$N_2$}\label{heur_figure_benchmark_N2}
\end{subfigure}
\begin{subfigure}[b]{.32\linewidth}
\includegraphics[width=\linewidth]{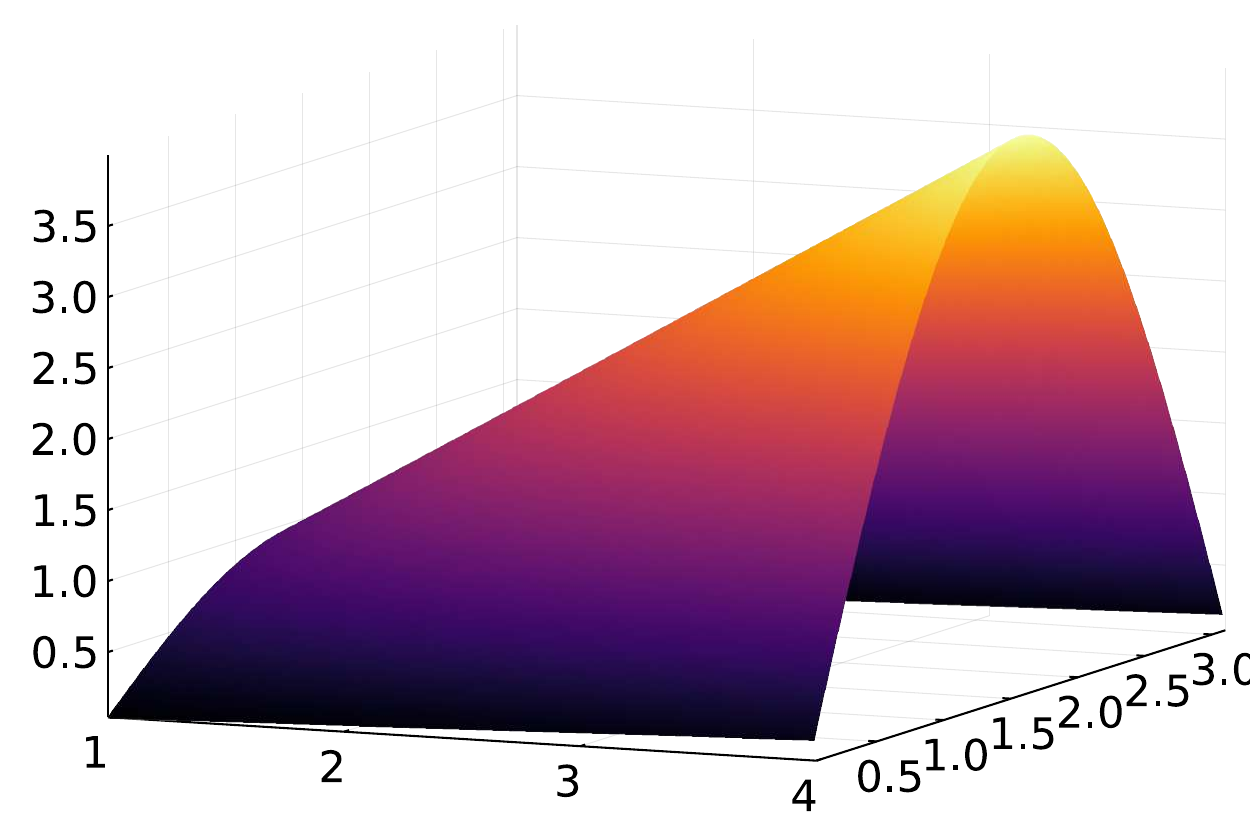}
\caption{$N_3$}\label{heur_figure_benchmark_N3}
\end{subfigure}
\begin{subfigure}[b]{.32\linewidth}
\includegraphics[width=\linewidth]{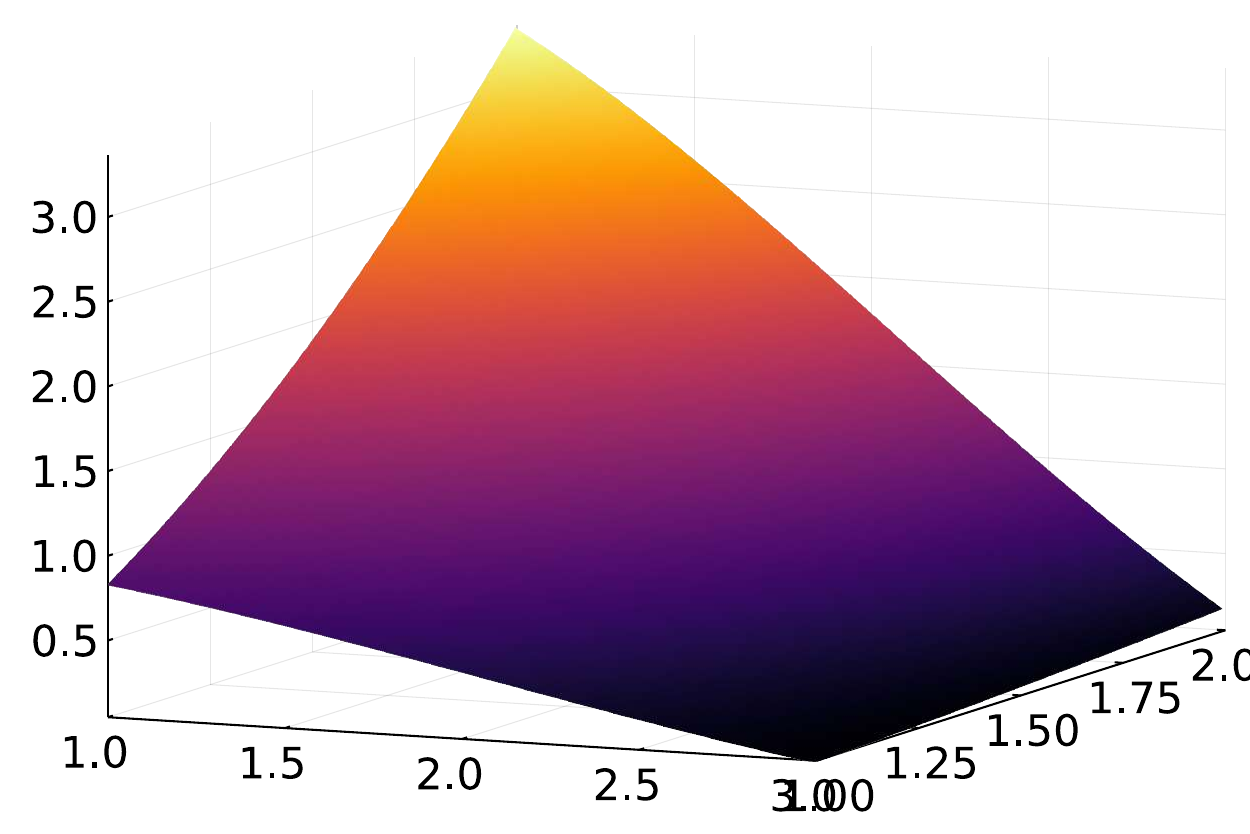}
\caption{$N_4$}\label{heur_figure_benchmark_N4}
\end{subfigure}

\begin{subfigure}[b]{.32\linewidth}
\includegraphics[width=\linewidth]{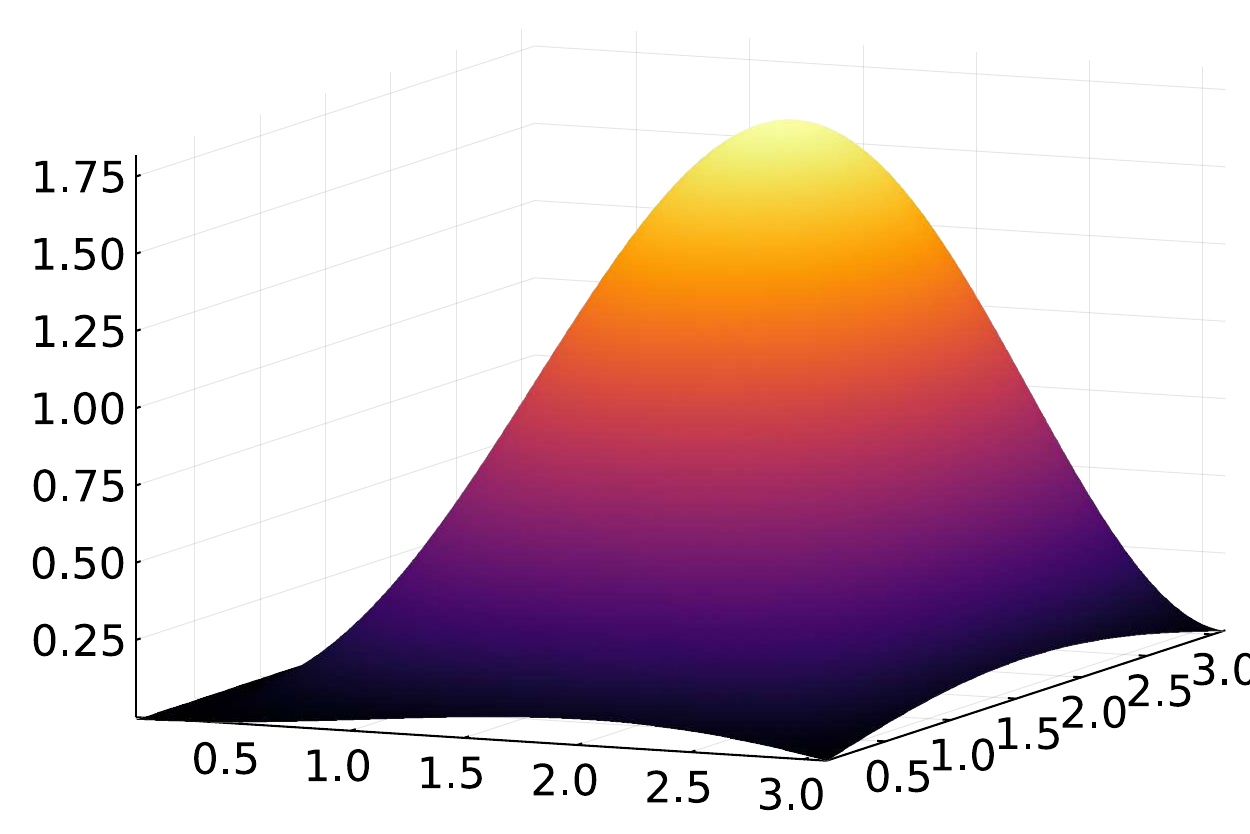}
\caption{$N_5$}\label{heur_figure_benchmark_N5}
\end{subfigure}
\begin{subfigure}[b]{.32\linewidth}
\includegraphics[width=\linewidth]{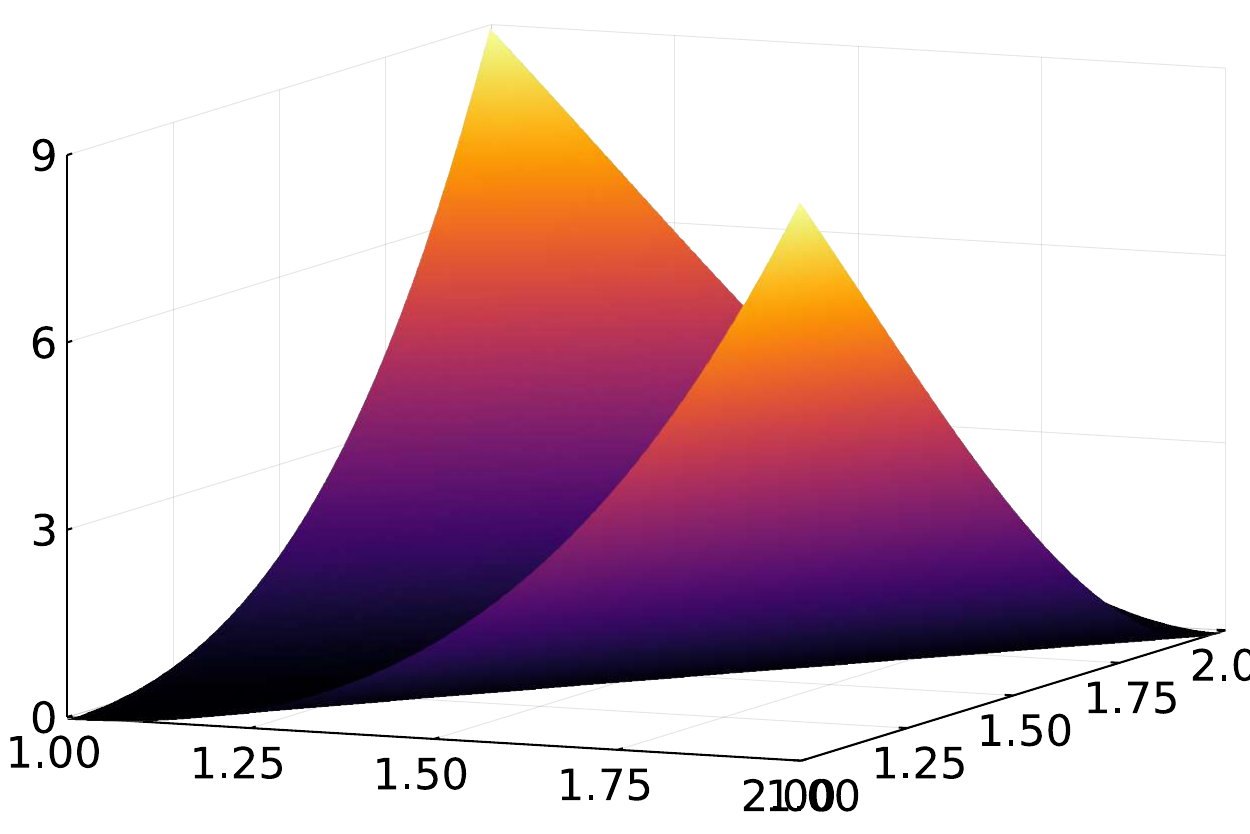}
\caption{$N_6$}\label{heur_figure_benchmark_N6}
\end{subfigure}
\begin{subfigure}[b]{.32\linewidth}
\includegraphics[width=\linewidth]{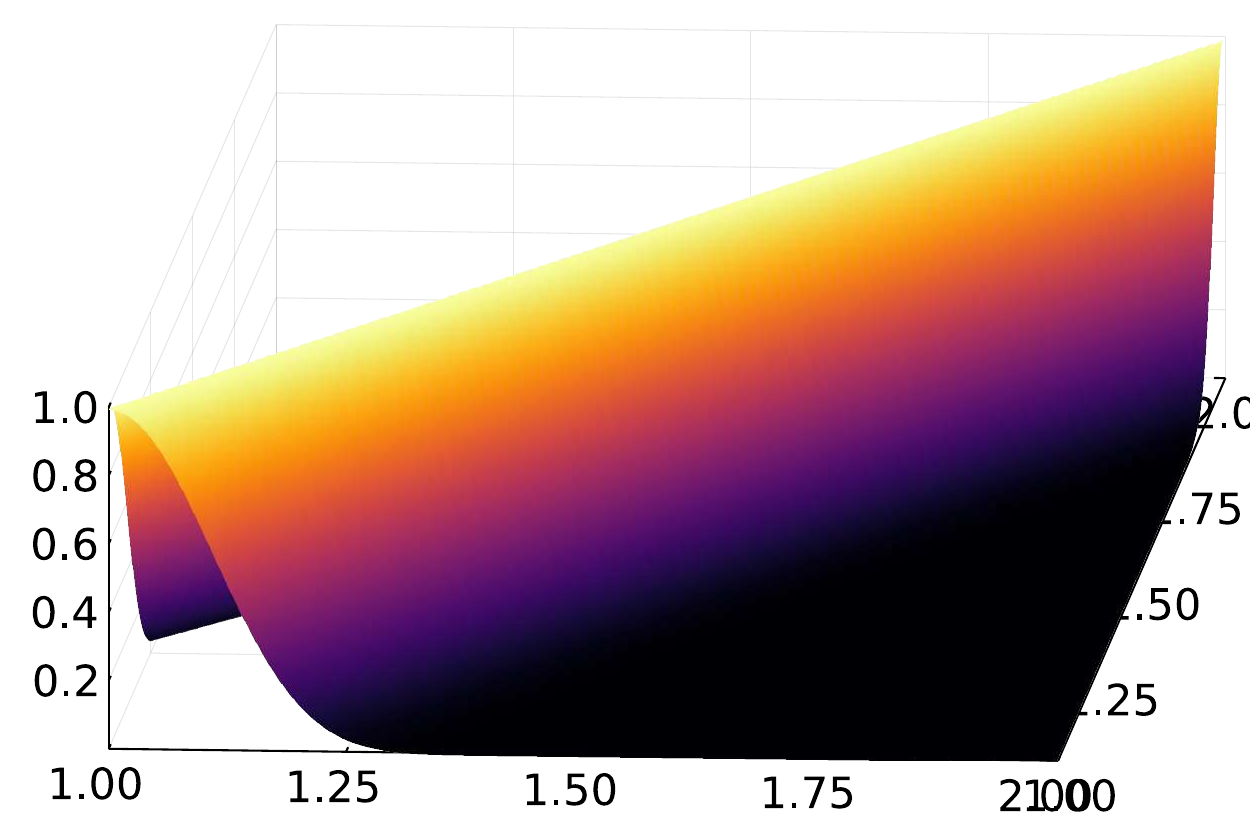}
\caption{$N_7$}\label{heur_figure_benchmark_N7}
\end{subfigure}
\caption{Functions approximated in the instance set}
\label{heur_figure_benchmark}
\end{figure}


The implementation is in Julia 1.12.5 \citep{Bezanson17julia}. We use the library JuMP 1.30.0 as modeling language for MILPs and LPs. GUROBI 13.0.1 is used for solving MILPs, while HiGHS 1.14.0 is used for solving LPs. All experiments were carried out on a single core of a 2.6 GHz Intel Xeon Gold 6126 processor, with 16 GB of RAM.

The Table \ref{table_parameter_values} presents the parameter values for the different algorithms.
\begin{table}[]
    \centering
    \begin{tabular}{ccc}
         name & value & used in \\ \midrule
         $\alpha$ & $2^{\frac{1}{3}}$ & \algmilp{}, \algcolor{}, \algmaximum{}, \algmaximal{} \\
         $p_0$ & $100$ & "\\
         $p_{add}$ & $20$ & "\\
         computation time & \SI{3600}{s} & " \\
         $k_{repetition}$ & $1$ & \algmilp{} and \algcolor{} \\
         $n_{MaximalCliques}$ & $100$ & \algmilp{} and \algmaximal{} \\
         $N_{samples}$ & $100$ & Computation of edges of two vertices, see Section \ref{subsec:construction_aretes_2} \\
         $n^S_{edge}$ & $16$ & Sampling the Domain for the $\mathbb{R}^2$-LDRFP \eqref{eq_R2-DRFP}, see Appendix \ref{sec:construction_S} \\
         $n^S_{int}$ & $200$ & "
    \end{tabular}
    \caption{Parameter values}
    \label{table_parameter_values}
\end{table}
In particular, preliminary experiments showed that the influence of the parameter $k_{repetition}$ taken in $\{0,1,4,16\}$ was small. It might be because too few were derived, or the hypergraph coloring constraints involving at least three vertices correspond to weak valid inequalities compared to hypergraph coloring constraints of two vertices, i.e., graph coloring constraints. A more detailed study of the strength of those valid inequalities is an interesting future research direction but is out of the scope of this work.

\subsection{Numerical Results from the Lower Bounding Algorithms} \label{subsection_results_lowerbounds}

Table \ref{table_all_relaxations} presents the lower bounds found by the algorithms \algmilp{}, \algcolor{}, \algmaximum{} and \algmaximal{} from section \ref{subsec:algo_relaxations}. The first row gives the name of the method. For \algmilp{}, the second row of the table specifies whether it is the complete algorithm or one of its two simplified variants ('-Preproc.' and '-Ineq.,-Preproc.') used to illustrate the influence of hypergraph coloring constraints on the quality of the bound.
Both variants employ the MILP formulation \eqref{model_DCFP} as a subalgorithm. However, unlike the complete algorithm, neither variant leverages maximal cliques to tighten the lower bounds before solving the MILP. As a result, constraints \eqref{MILP_fix_binary1} and \eqref{MILP_fix_binary0}, which fix specific binary variables based on maximal clique information, are excluded. This omission is reflected in the variants’ nomenclature by the suffix '-Preproc'. The variant labelled '-Ineq.,-Preproc.' further distinguishes itself from '-Preproc.' by excluding the hypergraph coloring constraints \eqref {MILP_edge_description}.

The remainder of Table \ref{table_all_relaxations} is organized as follows: the row labeled '\# seeds’ indicates the number of times each algorithm is executed for the same instance with different random seeds. Each algorithm is run once if there is no variability, and ten times if variability is present, which is the case for hyperedges generation or in the construction of a maximal clique. Subsequent rows present, for each instance, the (arithmetic) average lower bound obtained by the algorithms, with the maximum lower bound specified in parentheses when it is different from the average. The final row reports the geometric mean of the averaged results, along with the number of instances for which each algorithm achieves the best lower bound among the four compared methods. Instances marked with symbol $^\times$ were terminated prematurely due to the MILP solver reaching the memory limit.

\begin{table}
    \caption{Average lower bounds produced by the algorithms resulting from the four relaxations}
    \label{table_all_relaxations}
    \centering
    \scriptsize
    \begin{adjustbox}{max width=\textwidth}
    \begin{tabular}{cccccccc} \toprule
        ref & $\delta$ & \multicolumn{3}{c}{\algmilp{}} & \algcolor{} & \algmaximum{} & \algmaximal{} \\ 
        & & -Ineq.,-Preproc. & -Preproc. & complete & \\
        & \# seeds & 1 & 1 & 10 & 10 & 1 & 10 \\ \midrule
 & 1.5 & 3 & {\bf 4} & {\bf 4} & 3 & 3 & 3 \\ 
 & 1.0 & 3 & 5 & {\bf 6} & 5 & 4 & 5 \\ 
$L_1$ & 0.5 & 3 & 8 & {\bf 8.8} (9) & 8 & 7 & 8 \\ 
 & 0.25 & 3 & 10 & 12 & {\bf 13.7} (14) & 12 & 12 \\ 
 & 0.1 & 3 & 18 & 22 & 20.7 (23) & {\bf 23} & 21 \\ \midrule 
 & 1.5 & 3 & {\bf 5} & {\bf 5} & {\bf 5} & {\bf 5} & {\bf 5} \\ 
 & 1.0 & 3 & {\bf 6} & {\bf 6} & {\bf 6} & {\bf 6} & {\bf 6} \\ 
$L_2$ & 0.5 & 3 & 8 & {\bf 10} & {\bf 10} & 9 & 9 \\ 
 & 0.25 & 3 & 15 & 17 & {\bf 18} & 17 & 15 \\ 
 & 0.1 & 3 & 21 & 34 & 32 & {\bf 35} & 29 \\ \midrule 
 & 1.0 & {\bf 3} & {\bf 3} & {\bf 3} & 2 & 2 & 2 \\ 
 & 0.5 & 3 & {\bf 4} & {\bf 4} & 3 & 3 & 3 \\ 
$N_1$ & 0.25 & 3 & 5 & {\bf 6} & 5 & 4 & 4 \\ 
 & 0.1 & 3 & 7 & {\bf 9} & {\bf 9} & 7 & 8 \\ 
 & 0.05 & 3 & 12 & 12 & {\bf 13} & 11 & 12 \\ \midrule 
 & 0.1 & {\bf 1} & {\bf 1} & {\bf 1} & {\bf 1} & {\bf 1} & {\bf 1} \\ 
 & 0.05 & {\bf 2} & {\bf 2} & {\bf 2} & {\bf 2} & {\bf 2} & {\bf 2} \\ 
$N_2$ & 0.03 & {\bf 3} & {\bf 3} & {\bf 3} & {\bf 3} & {\bf 3} & {\bf 3} \\ 
 & 0.01 & 3 & 5 & {\bf 6} & 5 & 4 & 5 \\ 
 & 0.001 & 3 & 15 & {\bf 26} & 23 & {\bf 26} & 22 \\ \midrule 
 & 1.0 & {\bf 2} & {\bf 2} & {\bf 2} & {\bf 2} & {\bf 2} & {\bf 2} \\ 
 & 0.5 & {\bf 3} & {\bf 3} & {\bf 3} & {\bf 3} & {\bf 3} & {\bf 3} \\ 
$N_3$ & 0.25 & 3 & {\bf 4} & {\bf 4} & {\bf 4} & 3 & 3 \\ 
 & 0.1 & 3 & 6 & {\bf 8} & 7 & 6 & 7 \\ 
 & 0.05 & 3 & 9 & 10.9 (11) & {\bf 11} & 10 & 10 \\ \midrule 
 & 0.5 & {\bf 2} & {\bf 2} & {\bf 2} & 1 & 1 & 1 \\ 
 & 0.25 & {\bf 2} & {\bf 2} & {\bf 2} & {\bf 2} & {\bf 2} & {\bf 2} \\ 
$N_4$ & 0.1 & 3 & {\bf 4} & {\bf 4} & {\bf 4} & 3 & 3 \\ 
 & 0.05 & 3 & 5 & {\bf 6} & {\bf 6} & 4 & 5 \\ 
 & 0.03 & 3 & 6 & {\bf 8} & {\bf 8} & 6 & 7 \\ \midrule 
 & 1.0 & {\bf 1} & {\bf 1} & {\bf 1} & {\bf 1} & {\bf 1} & {\bf 1} \\ 
 & 0.5 & 3 & {\bf 4} & {\bf 4} & 3 & 3 & 3 \\ 
$N_5$ & 0.25 & 3 & {\bf 4} & {\bf 4} & {\bf 4} & {\bf 4} & {\bf 4} \\ 
 & 0.1 & 3 & 8 & 9 & {\bf 9.4} (10) & 8 & 8 \\ 
 & 0.05 & 3 & 12 & {\bf 14} & 13 & 13 & 13 \\ \midrule 
 & 1.0 & {\bf 3} & {\bf 3} & {\bf 3} & {\bf 3} & {\bf 3} & {\bf 3} \\ 
 & 0.5 & 3 & {\bf 4} & {\bf 4} & {\bf 4} & {\bf 4} & {\bf 4} \\ 
$N_6$ & 0.25 & 3 & {\bf 5} & {\bf 5} & {\bf 5} & {\bf 5} & {\bf 5} \\ 
 & 0.1 & 3 & 8 & {\bf 9} & 8 & 8 & 8 \\ 
 & 0.05 & 3 & 12 & 13.5 (14) & {\bf 15} & 12 & 14 \\ \midrule 
 & 1.0 & {\bf 1} & {\bf 1} & {\bf 1} & {\bf 1} & {\bf 1} & {\bf 1} \\ 
 & 0.5 & ~~{\bf 1}$^\times$ & {\bf 1} & {\bf 1} & {\bf 1} & {\bf 1} & {\bf 1} \\ 
$N_7$ & 0.25 & {\bf 3} & ~~1$^\times$ & {\bf 3} & {\bf 3} & {\bf 3} & {\bf 3} \\ 
 & 0.1 & 3 & {\bf 5} & {\bf 5} & 4.1 (5) & 4 & 4 \\ 
 & 0.05 & 3 & {\bf 7} & {\bf 7} & {\bf 7} & {\bf 7} & {\bf 7} \\ \midrule 
\multicolumn{2}{c}{geometric mean} & 2.625 & 4.636 & 5.206 & 4.876 & 4.549 & 4.642 \\
\multicolumn{2}{c}{\# best} & 13 & 24 & 37 & 29 & 20 & 17 \\ \bottomrule
    \end{tabular}
    \end{adjustbox}
\end{table}

\emph{Impact of the hypergraph coloring constraints.} Variant '-Preproc.' outperforms variant '-Ineq.,-Preproc.' for all instances except for one instance where '-Preproc.' ran out of memory.  The best lower bound obtained by '-Ineq.,-Preproc.' never exceeds 3, which is an order of magnitude lower than the best lower bound of 21 achieved by '-Preproc.'. These results confirm that incorporating hypergraph coloring constraints has a positive and substantial impact on the quality of the lower bounds obtained for the \RCFP{}. This empirical improvement likely extends to instances of the \RmCFP{} as well.

\emph{Overall performance of the lower bounding algorithms} Among the four lower bounding algorithms introduced in this paper, \algmilp{}  produces the best lower bound for 37 out of 45 instances, and achieves the highest geometric mean. The remaining algorithms, in descending order of the number of best lower bounds found, are: \algcolor{}, \algmaximum{} and \algmaximal{}. If we consider the geometric mean instead, the order is the same except that \algmaximum{} and \algmaximal{} are exchanged. It indicates that the former method finds best lower bounds among the four algorithms more often than the latter method, but also more often worse lower bounds than the latter. 
The results indicate that tighter relaxations generally produce superior lower bounds with the current time limit. However, in the eight cases where \algmilp{} does not perform best, $\delta$ is among the two smallest values. This suggests that, when $\delta$ is small, which results in larger models, the computational cost of solving the tightest relaxation models may be prohibitive, rendering simpler or smaller models more practical.

Finally, looking at the results instance by instance, remark that the maximum lower bound differ from the average only in few cases (7 out of 135), indicating that the variability due to the random choices is not important. Specifically, there is no variability in the results of \algmaximal{}.

\subsection{Numerical results from the Heuristics of the Literature} \label{subsection_results_heuristics}
\begin{table}
    \caption{Comparison of solution values of the \RCFP{} for methods of the literature with absolute approximation error $\delta$}
    \label{heur_table_comparaison_2DPWL_heuristics}
    \centering
    \scriptsize
    \resizebox{0.6\textheight}{!}{
    \begin{tabular}{cccccccccc}
         \toprule 
         ref & $\delta$ & DLN95 & DLN99 & RK2D & RK1D & \iLinA{} & KL21 & KL23-LP & KL23-LPrelax \\ \midrule
         & 1.5 & 6 & 6 & 16 & 6 & {\bf5} & 6 & 6 & 6 \\ 
         & 1.0 & 8 & 8 & 20 & 8 & {\bf6} & {\bf6} & 10 & 8 \\ 
        $L_1$ & 0.5 & 16 & 16 & 48 & 15 & {\bf12} & TL & 18 & 16 \\ 
         & 0.25 & 36 & 34 & 80 & 32 & {\bf21} & TL & 36 & 32 \\ 
         & 0.1 & 85 & 83 & 224 & 60 & {\bf55} & TL & 98 & 83 \\ \midrule 
         & 1.5 & 6 & 6 & 24 & 6 & {\bf5} & TL & 7 & 6 \\ 
         & 1.0 & 9 & 8 & 28 & 8 & {\bf6} & {\bf6} & 11 & 7 \\ 
        $L_2$ & 0.5 & 24 & 25 & 84 & 15 & {\bf12} & TL & 20 & 14 \\ 
         & 0.25 & 42 & 41 & 121 & 32 & {\bf21} & TL & TL & 25 \\ 
         & 0.1 & 120 & 118 & 351 & 60 & {\bf55} & TL & TL & TL \\ \midrule 
         \multicolumn{2}{c}{\# best} & 0 & 0 & 0 & 0 & {\bf10} & 2 & 0 & 1 \\ \bottomrule
    \end{tabular}}
    \vspace{0.2cm}

        \resizebox{0.6\textheight}{!}{
    \begin{tabular}{cccccccccc}
         \toprule 
         ref & $\delta$ & DLN95 & DLN99 & RK2D & RK1D & KL21 & KL23-LP & KL23-LPrelax \\ \midrule
            & 1.0 & {\bf3} & {\bf3} & 4 & 6 & 4 & 6 & 4 \\ 
         & 0.5 & {\bf6} & {\bf6} & 12 & 8 & {\bf6} & 10 & 10 \\  
        $N_1$ & 0.25 & {\bf12} & {\bf12} & 20 & 18 & {\bf12} & 16 & 14 \\ 
         & 0.1 & {\bf26} & {\bf26} & 59 & 45 & TL & 36 & 30 \\ 
         & 0.05 & 52 & {\bf49} & 94 & 98 & TL & 64 & 57 \\ \midrule 
         & 0.1 & {\bf1} & {\bf1} & 2 & 4 & {\bf1} & {\bf1} & {\bf1} \\ 
         & 0.05 & {\bf2} & {\bf2} & 6 & 9 & {\bf2} & 3 & 4 \\ 
        $N_2$ & 0.03 & {\bf4} & {\bf4} & 10 & 12 & {\bf4} & 6 & {\bf4} \\ 
         & 0.01 & 11 & {\bf10} & 31 & 30 & TL & 19 & 18 \\ 
         & 0.001 & 89 & {\bf87} & 350 & 270 & TL & TL & 181 \\ \midrule 
         & 1.0 & 3 & 3 & 5 & 12 & TL & {\bf2} & 3 \\ 
         & 0.5 & 4 & {\bf3} & 8 & 16 & TL & 6 & {\bf3} \\ 
        $N_3$ & 0.25 & {\bf5} & {\bf5} & 16 & 24 & 8 & 10 & 8 \\ 
         & 0.1 & 14 & {\bf13} & 44 & 72 & TL & 15 & 19 \\ 
         & 0.05 & 28 & {\bf27} & 85 & 125 & TL & 36 & {\bf27} \\ \midrule 
         & 0.5 & {\bf2} & {\bf2} & {\bf2} & 3 & {\bf2} & {\bf2} & {\bf2} \\ 
         & 0.25 & {\bf3} & {\bf3} & 4 & 10 & 4 & 4 & 4 \\ 
        $N_4$ & 0.1 & {\bf6} & {\bf6} & 9 & 28 & {\bf6} & {\bf6} & {\bf6} \\ 
         & 0.05 & 12 & {\bf11} & 23  & 27 & 12 & 14 & 14 \\ 
         & 0.03 & 18 & {\bf17} & 40 & 48 & TL & 25 & 26 \\ \midrule 
         & 1.0 & {\bf1} & {\bf1} & 6 & 20 & {\bf1} & {\bf1} & {\bf1} \\ 
         & 0.5 & 5 & 5 & 6 & 42 & {\bf4} & {\bf4} & {\bf4} \\ 
        $N_5$ & 0.25 & 12 & 14 & 21 & 80 & {\bf5} & 13 & 6 \\ 
         & 0.1 & 24 & 25 & 96 & 168 & TL & 24 & {\bf21} \\ 
         & 0.05 & 49 & 46 & 274 & 340 & TL & {\bf43} & 44 \\ \midrule 
         & 1.0 & {\bf3} & {\bf3} & 6 & 6 & {\bf3} & {\bf3} & {\bf3} \\ 
         & 0.5 & {\bf4} & {\bf4} & 9 & 7 & {\bf4} & {\bf4} & {\bf4} \\ 
        $N_6$ & 0.25 & 7 & 7 & 12 & 18 & {\bf6} & 14 & 10 \\ 
         & 0.1 & 14 & {\bf13} & 40 & 40 & TL & 22 & 16 \\ 
         & 0.05 & 27 & {\bf26} & 87 & 83 & TL & 50 & 38 \\ \midrule 
         & 1.0 & {\bf1} & {\bf1} & 2 & 3 & {\bf1} & {\bf1} & {\bf1} \\ 
         & 0.5 & {\bf2} & {\bf2} & 4 & 4 & {\bf2} & {\bf2} & {\bf2} \\ 
        $N_7$ & 0.25 & 4 & {\bf3} & 6 & 7 & 4 & 19 & 12 \\ 
         & 0.1 & 7 & 7 & 84 & 18 & {\bf5} & 19 & 8 \\ 
         & 0.05 & {\bf11} & {\bf11} & 86 & 34 & TL & 19 & 14 \\ \midrule
         \multicolumn{2}{c}{\# best} & 17 & {\bf28} & 1 & 0 & 16 & 11 & 13 \\ \bottomrule
    \end{tabular}}
\end{table}

Table \ref{heur_table_comparaison_2DPWL_heuristics} gives the number of pieces of feasible solutions found by methods of the literature or from Section \ref{sec:new_UB}. The first row gives the name of the method. DLN95 and DLN99 are the improvements of the method from \citet{Duguet22a} as described in Section \ref{sec:new_UB}, with bounding efficiency $\eta = 0.95$ and $0.99$ respectively. RK2D is Algorithm 2.1 from \citet{Rebennack15a}.  \iLinA{}, described in Section \ref{sec:new_UB}, is the improvement from the method LinA2D obtained by substituting the algorithm for \rCFP{} of \citet{Codsi25} in the algorithm of Section 3 of \citet{Rebennack15a} denoted RK1D.
KL21 is the method from \citet{Kazda21} and KL23-LP and KL23-LPrelaxed are the two methods from \citet{Kazda23}. Numbers in bold represent the best upper bound known for each instance. TL means that no feasible solutions were found in the time limit.

Results show that all of the best known solutions for linearly separable instances are produced by \iLinA{} whereas 80\% of the best known solutions for nonlinearly separable instances, are produced by DLN99.

\subsection{Optimal Solutions of the \RCFP{}} \label{subsection_optimal_solutions}

This work describes the first algorithms for establishing lower bounds to the \RmCFP{}.
Combining the lower bounds from Section \ref{subsection_results_lowerbounds} and upper bounds from Section \ref{subsection_results_heuristics} allows to prove for the first time that some nontrivial solutions are optimal solutions of the \RCFP{}. Table \ref{table_LB_UB} shows the best lower and upper bounds found in Tables \ref{table_all_relaxations} and \ref{heur_table_comparaison_2DPWL_heuristics}. If the upper bound is equal to the lower bound then the instance is solved to optimality. In this case, the upper bound is in bold and is followed by a *. If the difference between the upper and lower bound is equal to 1, a $\dagger$ symbol is added after the upper bound. %

\begin{table}
    \caption{Best known lower and upper bounds of the \RCFP{} \\
    symbols: *: the instance is solved to optimality; $\dagger$ the solution is proven to be at most one unit far from the optimum.}
    \label{table_LB_UB}
    \scriptsize
    \centering
    \begin{tabular}{ccrl} \toprule
        f & $\delta$ & $\mathit{LB}$ & $\mathit{UB}$\\ \midrule
         & 1.5 & 4 & ~$5^{\dagger}$ \\ 
         & 1.0 & \textbf{6} & ~$\mathbf{6^*}$ \\ 
        $L_1$ & 0.5 & 9 & 12 \\
         & 0.25 & 14 & 21 \\ 
         & 0.1 & 23 & 55 \\ \midrule 
         & 1.5 & \textbf{5} & ~$\mathbf{5^*}$ \\ 
         & 1.0 & \textbf{6} & ~$\mathbf{6^*}$ \\ 
        $L_2$ & 0.5 & 10 & 12 \\ 
         & 0.25 & 18 & 21 \\ 
         & 0.1 & 35 & 55 \\ \midrule 
         & 1.0 & \textbf{3} & ~$\mathbf{3^*}$ \\ 
         & 0.5 & 4 & ~6 \\ 
        $N_1$ & 0.25 & 6 & 12 \\ 
         & 0.1 & 9 & 26 \\ 
         & 0.05 & 13 & 49 \\ \midrule 
         & 0.1 & \textbf{1} & ~$\mathbf{1^*}$ \\ 
         & 0.05 & \textbf{2} & ~$\mathbf{2^*}$ \\ 
        $N_2$ & 0.03 & 3 & ~$4^{\dagger}$ \\ 
         & 0.01 & 6 & 10 \\ 
         & 0.001 & 26 & 87 \\ \midrule 
         & 1.0 & \textbf{2} & ~$\mathbf{2^*}$ \\ 
         & 0.5 & \textbf{3} & ~$\mathbf{3^*}$ \\ 
        $N_3$ & 0.25 & 4 & ~$5^{\dagger}$ \\ 
         & 0.1 & 8 & 13 \\ 
         & 0.05 & 11 & 27 \\ \midrule 
         & 0.5 & \textbf{2} & ~$\mathbf{2^*}$ \\ 
         & 0.25 & 2 & ~$3^{\dagger}$ \\ 
        $N_4$ & 0.1 & 4 & ~6 \\ 
         & 0.05 & 6 & 11 \\ 
         & 0.03 & 8 & 17 \\ \midrule 
         & 1.0 & \textbf{1} & ~$\mathbf{1^*}$ \\ 
         & 0.5 & \textbf{4} & ~$\mathbf{4^*}$ \\ 
        $N_5$ & 0.25 & 4 & ~$5^{\dagger}$ \\ 
         & 0.1 & 10 & 21 \\ 
         & 0.05 & 14 & 43 \\ \midrule 
         & 1.0 & \textbf{3} & ~$\mathbf{3^*}$ \\ 
         & 0.5 & \textbf{4} & ~$\mathbf{4^*}$ \\ 
        $N_6$ & 0.25 & 5 & ~$6^{\dagger}$ \\ 
         & 0.1 & 9 & 13 \\ 
         & 0.05 & 15 & 26 \\ \midrule 
         & 1.0 & \textbf{1} & ~$\mathbf{1^*}$ \\ 
         & 0.5 & 1 & ~$2^{\dagger}$ \\ 
        $N_7$ & 0.25 & \textbf{3} & ~$\mathbf{3^*}$ \\ 
         & 0.1 & \textbf{5} & ~$\mathbf{5^*}$ \\ 
         & 0.05 & 7 & 11 \\ \bottomrule 

    \end{tabular}
\end{table}

In total, $16$ instances were proven optimal, including $13$ with a nontrivial lower bound strictly greater than $1$. In addition, for 7 instances the difference between the lower and the upper bound is 1. Thus 23 instances out of 45 are solved or 'the closest' from being solved. The optimal solutions come mostly from ``small'' instances, i.e., they have an optimal solution with at most $6$ pieces. Of course, a mathematical analysis of each instance could probably prove the optimality of some more of them. 
However, the objective of our work is to find optimal solutions with a generic algorithm instead of a theoretical function-specific analysis.

The largest bound ratio $\mathit{UB}/\mathit{LB}$ is $49/13 \approx 3.8$ and occurs for the instance with function $N_1(x,y) = xy$ and $\delta=0.05$.
It would be interesting to find a nonlinear function for which we know the minimum number of pieces for very different values of $\delta$ in order to observe the differences with the lower and upper bounds. This could provide useful insights for their improvement, as well as an idea of whether it is the upper bound or the lower bound that is the furthest away from the optimal value.

While the theoretical foundation underlying the lower bounding algorithms is generic, our current implementations are restricted to the \RCFP{}. A generic implementation is a potential future direction, but only after addressing scalability issues. Indeed, in the two-variable case (i.e., dimension $2$), we have proven optimality for small instances (i.e., with up to six pieces).
Moving to higher dimensions usually requires a greater number of linear pieces to maintain the same final precision. In addition, an increase in dimensionality may necessitate an increase in the number of points required for the domain discretization. This leads to larger hypergraphs and, consequently, larger subproblems that must be solved at each iteration. Therefore, improving scalability is a prerequisite before an implementation to address a dimension greater than two.

\section{Conclusion}
\label{s:theconclusion}
This paper investigates the first lower bounds for the Corridor Fitting Problem as well as new upper bounds for the two-variable case. The lower bounds are based on a discretization of the corridor domain, and on the concept of not linearly compatible sets: sets of points of the domain that cannot be covered by the same linear piece. Each not linearly compatible set corresponds to a valid inequality for the problem. It allows us to derive a relaxation that exhibits the structure of a hypergraph coloring problem but with a number of hyperedges exponential in the number of nodes. It is leveraged to derive four further relaxations that can be sorted depending on the amount of information of the original Corridor Fitting Problem that is used, and give rise to four different lower bounding algorithms for the problem.

Experimental comparisons on the classic instances of the literature on the Corridor Fitting Problem for functions of two variables established that the more information the algorithm uses, the better lower bound it produces in average. Finally, we compared the lower bounds obtained with the best upper bound from the literature or from this work. This allowed to prove that for more than half of the instances, the solution is optimal or is at most one unit away from the optimum.

Improving scalability is a prerequisite to the viability of implementing algorithms for dimensions greater than two.
We propose two directions to improve the computations of lower bounds for the Corridor Fitting Problem. First, adding points in the domain discretization during optimization as in \citet{Kazda21} looks promising. 
Second, a branch-and-cut framework should make a better use of the valid inequalities that are the hypergraph coloring constraints, as only those cutting off an infeasible solution for the Corridor Fitting Problem would be added.

\section*{Acknowledgements}

The first author has been funded by a PhD grant from the Ministère de l'Enseignement Supérieur et de la Recherche in France (MESR) as well as a grant from the Deutsche Forschungsgemeinschaft
(DFG) in the project 543678993 (Aggregative gemischt-ganzzahlige
Gleichgewichtsprobleme: Existenz, Approximation und Algorithmen) in Germany.
The third author have been funded by the Responsible Retail Chair (Chaire de formation et de recherche retail responsable. \url{https://sites.laas.fr/projects/ChaireRetailResp/}).
The computations were executed on the high performance cluster
``Elwetritsch'' at the TU Kaiserslautern, which is part of the
``Alliance of High Performance Computing Rheinland-Pfalz'' (AHRP).
We kindly acknowledge the support of the MESR, the DFG, the Responsible Retail Chair and the AHRP.

\bibliographystyle{plainnat}
\bibliography{article_library.bib}

\appendix
\section{Complexity Proof of Proposition \ref{prop_complexity_relaxations}} \label{section_complexity_proof}

\begin{proof} \label{proof_complexity_relaxations}
    We start by showing that \relaxRonestrong{} is NP-hard. To do so, we rely on the fact that the decision variant of the graph coloring problem with at least 3 colors is NP-complete.
    To prove that \relaxRonestrong{} is NP-hard it is therefore sufficient to show that the graph coloring problem can be reduced to \relaxRonestrong{}. For the reduction, we have to show that given an instance of the graph coloring problem described by the graph $G=(V,E)$, there exists an instance $I$ of \relaxRonestrong{} so that a solution of instance $I$ corresponds to a solution of the instance of graph coloring. 
    The instance $I$ of \relaxRonestrong{} is described by:
    \begin{itemize}
        \item a corridor $\C = \mathit{Corridor}(u,l,D)$,
        \item a graph $H=(X,\hat{E})$ that is a (potentially incomplete) incompatibility hypergraph of $\C$, 
        \item a set $X \subset D$ with $|X|=|V|$.
    \end{itemize}
    We now build such an instance $I$.
    Let $m = |V|$.
    Take domain $D = \{x\geq 0 | \sum_{k=1}^m x_k = 1\}$ to be the $m-1$-dimensional simplex. Choose $X$ to contain, for each $i=1,\ldots,m$, the point $x^i$ which is the vector with $0$ coefficients everywhere except a coefficient $1$ in coordinate $i$. Also, define the continuous functions $l$ and $u$ such as: $l(x) = \lVert x \rVert = u(x)-\epsilon$.
    For a small enough $\epsilon$, there is no linear function fitting the $\mathit{Corridor}(u,l,conv(\{x^i,x^j\}))$ for any $x^i,x^j \in X$ with $x^i \neq x^j$ and thus $\{x^i,x^j\}$ is a not linearly compatible set. 
    Finally, define $H$ to be the graph with vertices $X$ and edges $(x^i,x^j)$ if $(i,j)$ is an edge of $G$. Here, we use that $H$ needs not be a complete incompatibility hypergraph. Remark that the graphs $H$ and $G$ are isomorphic.
    
    It remains to show that an optimal solution of the instance of \relaxRonestrong{} corresponds to an optimal solution of the instance of graph coloring. Consider an assignment of colors to points of $X$ satisfying the hypergraph coloring constraints of \relaxRonestrong{} and minimizing the number of colors $n^*$. We denote by $S_j$ the set of points of $X$ assigned to color $j$. We show that associating the linear function $L_j(x) = \sum_{j \in S_j} x_j$ to color $j$ for any $j$ satisfies all the other constraints of \relaxRonestrong{}.
    Constraint \ref{constraint_valid_inequalities} is satisfied because $L_j$ is a linear function. Constraint \ref{constraint_non_intersecting_domains} is satisfied because the domains $D_j = conv(\{x^i \in X | x^i \text{ is of color } j\}$ are of disjoint support. Indeed, consider $y \in D_j$ and $z \in D_{j'}$, $j \neq j'$. The non-zero components of $y$ and $z$ do not have the same indices; thus $y \neq z$. Finally, Constraint \ref{constraint_def_discrete_pointwise} is satisfied because $l(x^i) = 1 = L_j(x^i) < u(x^i)$ for all $i=1,\ldots,m$ and $j=1,\ldots,n^*$. Thus $n^*$ is the optimal value of \relaxRonestrong{}.
    Since the graphs $H$ and $G$ are isomorphic, the corresponding assignment of colors in $G$ is also optimal.
    
    We now prove the other three statements.
    Regarding the proof of NP-hardness of \relaxRone{}, it suffices to show that the graph coloring problem can be reduced to \relaxRone{}. This proof can be adapted from the proof of complexity of \relaxRonestrong{}. Indeed, it is sufficient to consider the same construction of $\mathit{Corridor}(u,l,D)$, graph $H$ and linear functions $L_j$ without the part on the constraints that are not hypergraph coloring constraints.
    As for \relmaximum{}, it suffices to show that the maximum clique problem can be reduced to \relmaximum{}. This proof can also be adapted from the one of \relaxRonestrong{}.
    Finally, \relmaximal{} is in P because it is a special case of the maximal clique problem, which is in P.
\end{proof}

\section{Details on the Lower Bounding Algorithms Leveraging the Relaxations} \label{sec:detail_algorithms}

The algorithms leveraging the relaxations have a similar design. We first present the simplest of the four algorithms while for the others we mainly highlight the differences to the first. Recall that the pseudo-codes are generic but our implementations for the computational experimentation are specific to the two-variable case. 

\subsection{Lower Bounding Algorithm Leveraging Relaxation \relmaximum{} } \label{subsec:algo_clique_maximum}
The algorithm \algmaximum{} computes a lower bound using the relaxation \relmaximum{} for increasingly large incompatibility graphs until the computation time limit is reached. It is described in Algorithm \ref{algmaximum}.

Each iteration starts with the computation of the number of vertices $p_{vertices}$. This number scales affinely with the number of pieces (i.e., the value $\mathit{LB}$) and is multiplied by a coefficient $\alpha$ after every iteration that did not lead to an improvement of the value of LB. The algorithm then generates the incompatibility graph $G=(X,E_2)$ where $X$ is a set of $p_{vertices}$ points of the domain $D$ of the $\mathit{Corridor}(u,l,D)$ computed as explained in Section \ref{subsec:construction_vertices} whereas $E_2$ is the set of size 2 incompatibility edges obtained as explained in Section \ref{subsec:construction_aretes_2}. 
The method \textsc{ComputeMaximumClique} in Line \ref{algo_maximumclique_computemaxclique} computes a maximum clique of the graph $(X,E_2)$ with a state-of-the-art method from \citet[Appendix A.4]{Walteros20}. It is the variant using a binary search of the exact algorithm for the maximum clique problem. It is called with a time limit equal to the remaining time in \algmaximum{}.
The cardinality of the maximum clique found is a valid lower bound of the \RmCFP{}. The best known lower bound is updated if there is an improvement. Otherwise the vertex parameter $k_{optimal}$ is incremented so that the number of vertices in the next iteration is increased.
\begin{algorithm}[!ht]
\textbf{Input:} \\
- $\mathcal{C}=\mathit{Corridor}(u,l,D)$ the instance of the \RmCFP{}  \\
- coefficients $\alpha > 1, p_0, p_{add}$ defining the number of vertices of the incompatibility graph \\
- the computation time limit of the algorithm \\
\textbf{Output:} $\mathit{LB}$ a lower bound on the value of the \RmCFP{} for $\mathit{Corridor}(u,l,D)$

\noindent~Initialization
\begin{algorithmic}[1]

\State $\mathit{LB} = 1$, $k_{optimal} = 0$

\end{algorithmic}

\begin{algorithmic}[1]
\setcounter{ALG@line}{1}
\While{time remaining $> 0$} \label{algo_R2maximum_while}

\noindent \hspace{-0.5cm}Build the graph

    \State $p_{vertices} = \lceil(p_0 + p_{add} \mathit{LB}) \times \alpha^{k_{optimal}}\rceil$ \label{algo_calcul_N_R2maximum}
    \State $X = \textsc{ComputeVertices}(\mathcal{C},p_{vertices})$ \Comment{see Section \ref{subsec:construction_vertices}} \label{algo_calcul_vertices_R2maximum} 
    \State $E_2 = \textsc{ComputeEdgesOfSize2}(\mathcal{C},X)$ \Comment{see Section \ref{subsec:construction_aretes_2}} \label{algo_function_aretes2_R2maximum}

\noindent \hspace{-0.5cm}Solve the relaxation 

    \State $c = \textsc{ComputeMaximumClique}(X,E_2,\text{time remaining})$ \label{line_maximum_clique} \label{algo_maximumclique_computemaxclique} \Comment{see \citep[Appendix A.4]{Walteros20}}

\noindent \hspace{-0.5cm}Update LB or increment the vertex parameter

    \If{$|c| > \mathit{LB} $}
        \State $\mathit{LB} = |c|$ \Comment{update LB}
    \Else
        \State $k_{optimal} = k_{optimal} + 1$ \Comment{increment the vertex parameter}
    \EndIf
\EndWhile
\end{algorithmic}
Return the best lower bound found
\begin{algorithmic}[1]
\setcounter{ALG@line}{12}
\State $\text{Return } \mathit{LB}$
\end{algorithmic}
\caption{\algmaximum{}}
\label{algmaximum}
\end{algorithm}

\subsection{Lower Bounding Algorithm Leveraging Relaxation \relmaximal{} } \label{subsec:algo_clique_maximal}

The algorithm \algmaximal{} computes a lower bound using the relaxation \relmaximal{} for increasingly large incompatibility graphs until the computation time limit is reached. The resulting Algorithm is similar to \algmaximum{}, except for the replacement of Line \ref{algo_maximumclique_computemaxclique} by the following line.
\begin{algorithm}[H]
\noindent Solve the relaxation 
\begin{algorithmic}[1]
\setcounter{ALG@line}{5}
\State $c = \textsc{MaximalCliqueHeuristic}(E_2,n_{maximalCliques})$
\end{algorithmic}
\end{algorithm}
\label{subsec:maximalcliqueheuristic}
Method \textsc{MaximalCliqueHeuristic}$(E_2,n_{maximalCliques})$ computes a maximal clique in a graph with edges $E$. It starts with an empty set of vertices and iteratively expands it with a vertex connected to all vertices already in the set of vertices and having the most edges (ties are broken at random). The process is stopped when no more vertices can be added.
This process is run $n_{maximalCliques}$ times and the largest clique found is returned.

\subsection{Lower Bounding Algorithm Leveraging Relaxation \relaxRone{}} \label{subsec:algo_gcbcfp}
The algorithm \algcolor{} computes a lower bound using the relaxation \relaxRone{} for increasingly large hypergraphs until the computation time limit is reached. It is described in Algorithm \ref{algcolor}.
\begin{algorithm}[!ht]
\textbf{Input:} \\
- $\mathcal{C}=\mathit{Corridor}(u,l,D)$ the instance of the \RmCFP{}  \\
- coefficients $\alpha > 1, p_0, p_{add}$ defining the number of vertices of the hypergraph \\
- the number of attempts $k_{repetition}$ to create edges connecting more than two vertices for each vertex of the hypergraph $(X,E)$ \\
- the computation time limit of the algorithm \\
\textbf{Output:} $\mathit{LB}$ a lower bound on the value of the \RmCFP{} for $\mathit{Corridor}(u,l,D)$

\noindent Initialization
\begin{algorithmic}[1]
\State $\mathit{LB} = 1$, $k_{optimal} = 0$
\While{time remaining $> 0$}

\noindent \hspace{-0.5cm}Build the incompatibility hypergraph

    \State $p_{vertices} = \lceil(p_0 + p_{add} \mathit{LB}) \times \alpha^{k_{optimal}}\rceil$ \label{algo_N_gc}
    \State $X = \textsc{ComputeVertices}(\mathcal{C},p_{vertices})$ \Comment{see Section \ref{subsec:construction_vertices}}
    \State $E_2 = \textsc{ComputeEdgesOfSize2}(\mathcal{C},X)$ \Comment{see Section \ref{subsec:construction_aretes_2}}
    \State $E_{3^+} = \textsc{ComputeEdgesOfSize3$^+$}(\C,X,E_2,p_{vertices},k_{repetition})$ \Comment{see Section \ref{subsec:construction_aretes_multi}}

\noindent \hspace{-0.5cm}Solve the relaxation 
    
    \State $c = \textsc{HypergraphColoring}(X,E_2\cup E_{3^+},\text{time remaining})$ \label{algo_gc-cdcl} \Comment{see Appendix \ref{subsubsec:hypergraphcoloringproblem}}

\noindent \hspace{-0.5cm}Update LB or increment the vertex parameter
    
    \If{$|c| > \mathit{LB} $}
        \State $\mathit{LB} = |c|$ \Comment{update LB}
    \Else
        \State $k_{optimal} = k_{optimal} + 1$ \Comment{increment the vertex parameter}
    \EndIf
\EndWhile
\end{algorithmic}
Return the best lower bound found
\begin{algorithmic}[1]
\setcounter{ALG@line}{13}
\State $\text{Return } \mathit{LB}$
\end{algorithmic}
\caption{\algcolor{} solving relaxation \relaxRone{}}
\label{algcolor}
\end{algorithm}

\subsubsection{Principle}
Similarly to \algmaximum{}, each iteration of \algcolor{} starts with the computation of the number of vertices $p_{vertices}$. The algorithm then generates the incompatibility hypergraph $H=(X,E_2 \cup E_{3^+})$ where $X$ is a set of $p_{vertices}$ points of the domain $D$ of the $\mathit{Corridor}(u,l,D)$, $E_2$ is the set of size 2 incompatibility edges obtained as explained in Section \ref{subsec:construction_aretes_2}, and $E_{3^+}$ is the set of incompatibility edges of size greater or equal to 3, obtained as explained in Section \ref{subsec:construction_aretes_multi}. The hypergraph coloring of $H$ is computed as described in Appendix \ref{subsubsec:hypergraphcoloringproblem}. Its cardinality is a valid lower bound of the \RmCFP{}. The best known lower bound is updated if there is an improvement. Otherwise the vertex parameter $k_{optimal}$ is incremented.

\subsubsection{Method Solving a Hypergraph Coloring Problem} \label{subsubsec:hypergraphcoloringproblem}
This method computes deterministically the chromatic number of the hypergraph $H=(X,E)$ with a time limit equal to the remaining time allocated to the algorithm. It is an adaptation to hypergraph coloring of the graph coloring algorithm $gc-cdcl$ described in \citet{Hebrard20}. It combines constraint programming \citep{Vanhentenryck03} and constraint satisfaction problem techniques \citep{Ghedira13CP} to find the chromatic number of a graph. 
The main modification to adapt the algorithm to hypergraph coloring is to remove, when there is at least one hyperedge, the upper bound computation specific to (non-hypergraph) graph coloring. Another modification is that when a new lower bound of the chromatic number is proven it is saved so that it can be recovered if the algorithm reaches the time limit.

\subsection{Lower Bounding Algorithm Leveraging Relaxation \relaxRonestrong{}} \label{subsection_alg_milp}
Let \relaxRonestrongdecision{} be the decision version of \relaxRonestrong{} for $n$ pieces and hypergraph $H$. The algorithm \algmilp{} produces lower bounds of the \RmCFP{} by solving \relaxRonestrongdecision{} for an increasing number of pieces $n$ and increasingly large hypergraphs. Each \relaxRonestrongdecision{} proven infeasible induces a lower bound $\mathit{LB}= n+1$ of the \RmCFP{}. The remainder of this section first presents the MILP \eqref{model_DCFP} modeling \relaxRonestrongdecision{}, then details \algmilp{}, referenced as Algorithm \ref{algorithm_MILP_GC}, that produces the lower bound by iteratively solving MILPs.

\subsubsection{MILP model \eqref{model_DCFP} for the \relaxRonestrongdecision{}} \label{subsubsec:MILP_model} 

Before describing the MILP model in detail, we give an overview of it. A solution to the MILP represents a \PWL{} function $g$ defined on a subset $D' \subset D$. Function $g$ corresponds to a union of polytopes of $\mathbb R^{m+1}$ (with points $(x,g(x))$ for any $x \in D'$) satisfying the corridor constraints on a finite number of points $X \subset D'$. To that end, there are constraints ensuring a valid $n$-coloring of hypergraph $H$;
some constraints ensure that the corridor constraints are satisfied on $X$;
there are also constraints ensuring the domains of pieces of the generated \PWL{} function are convex and of disjoint interior.
\begin{figure}
    \centering
    \includegraphics[width=0.6\textwidth]{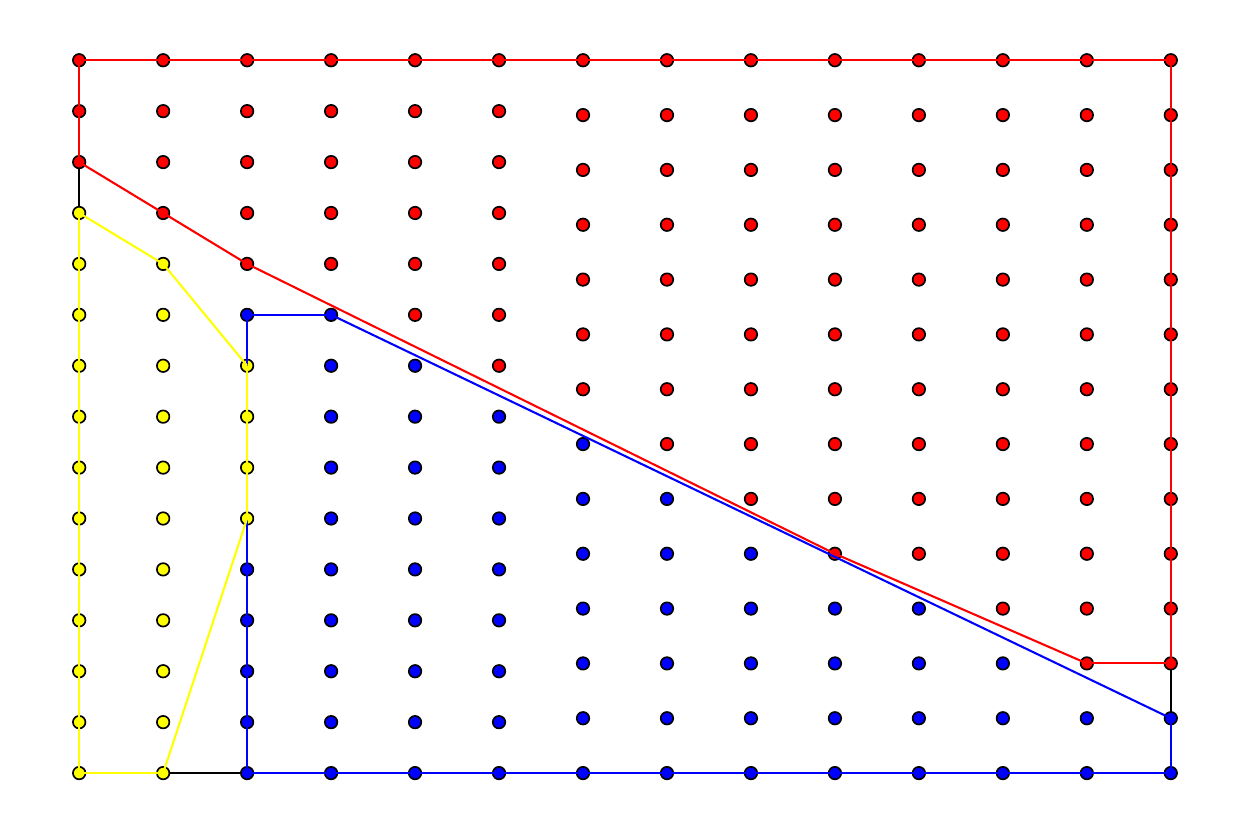}
    \caption{Example of solution of model \eqref{model_DCFP}}
    \label{figure_MILP_coloration_example}
\end{figure}
Figure \ref{figure_MILP_coloration_example} shows an example of solution of the MILP model \eqref{model_DCFP} for $n=3$ colors. The vertices of the hypergraph $H$ are colored yellow, blue or red depending on the piece containing them. The convex hull of the vertices of one color indicates the domain of one piece. The union of those convex hulls form the domain $D' \subset D$.

We now describe the MILP model more in detail.
One of the critical constraints of the MILP \eqref{model_DCFP} is that the domains of pieces of the generated \PWL{} function are convex and of disjoint interior. To that end, we enforce the domains of each pair of pieces to be linearly separable, i.e., there exist a hyperplane separating them.
\begin{definition}[linear separation]
Let $A$,$B \subset \mathbb R^m$. The sets $A$ and $B$ are \textit{linearly separable} if there exist a hyperplane $\alpha x = \beta$, $\alpha \in \mathbb R^m \setminus 0$ and $\beta \in \mathbb R$ such that $A$ is in the half-space $\alpha x \leq \beta$ and $B$ is in the half-space $\alpha x \geq \beta$.
\end{definition}

\begin{table}
    \centering
    \caption{Parameters of the MILP model}
    \label{MILP_parameters}
    \begin{adjustbox}{max width=\textwidth}
    \begin{tabular}{ll} \toprule
        parameters & purpose \\ \midrule
        $\mathit{Corridor}(u,l,D)$ & corridor of the \RmCFP{} \\ 
        $p_{vertices}$ & number of vertices of the hypergraph $H=(X,E)$ \\ 
        n & number of colors to use in hypergraph coloring \\ 
        $x,y$ & vectors containing the positions in $(x_i,y_i) \in D$ of the vertices of $H$ \\ 
        A & set containing all computed incompatibility edges \\ 
        $MaximalClique$ & maximal clique computed with the edges connecting two vertices of $H$ \\ \bottomrule
    \end{tabular}
    \end{adjustbox}
\end{table}

\begin{table}
    \centering
    \caption{Variables of the MILP model}
    \label{MILP_variables}
    \resizebox{\textwidth}{!}{
    \begin{tabular}{lll} \toprule
        variable & indices & purpose \\ \midrule
        \multirow{2}{*}{$b_{i,j} \in \{0,1\}$} & \multirow{2}{*}{$i = 1,...,p_{vertices}$, $j = 1,...,n$} & \multirow{2}{*}{equals $1$ if point $i$ is covered by piece $j$} \\
        &&  \\ 
        $\alpha_j \in \mathbb R^m,$ & \multirow{2}{*}{$j = 1,...,n$} & coefficients of the linear function of piece $j$ \\
        $\beta_j \in \mathbb R$ && with expression $\alpha_j x + \beta_j$ \\ 
        \multirow{2}{*}{$a^+_{j,j'} \in \mathbb R^m_+$} & \multirow{2}{*}{$j,j' = 1,...,n,j < j'$} & positive part of the normal vector for the \\
        && linear separation of pieces $j$ and $j'$ \\ 
        \multirow{2}{*}{$a^-_{j,j'} \in \mathbb R^m_+$} & \multirow{2}{*}{$j,j' = 1,...,n,j < j'$} & negative part of the normal vector for the \\
        && linear separation of pieces $j$ and $j'$ \\ 
        \multirow{2}{*}{$s^k_{j,j'} \in \{0,1\}$} & $j,j' = 1,...,n,j < j',$  &  sign of the coefficient $(a_{j,j'})_k$ of the normal \\
        & $k = 1,...,m$ & vector for the linear separation \\ 
        \multirow{2}{*}{$w_{j,j'} \in \mathbb R$} & \multirow{2}{*}{$j,j' = 1,...,n,j < j'$}  & right-hand side of the linear separation of \\
        && pieces $j$ and $j'$ \\ \bottomrule
    \end{tabular}}
\end{table}

Table \ref{MILP_parameters} describes the parameters used in the MILP model. 
Table \ref{MILP_variables} lists the variables defined in the model. It specifies all the indices, the type, and the role of the variables in the model. To make the model easier to read, we sometimes write constraints with vectors. In addition, we use the following notation:
\begin{align}
    & (a_{j,j'})_k \define (a^+_{j,j'})_k - (a^-_{j,j'})_k \quad \forall k = 1,\dots,m.
\end{align}
 The Model \eqref{model_DCFP} is as follows.

\begin{subequations} \label{model_DCFP}
\begin{align}
    & (b_{i,j} = 1) \Rightarrow l(x_i) \leq \alpha_j x_i + \beta_j \quad \forall i = 1,...,p_{vertices}, \forall j = 1,...,n, \label{MILP_respect_l} \\
    & (b_{i,j} = 1) \Rightarrow u(x_i) \geq \alpha_j x_i + \beta_j \quad \forall i = 1,...,p_{vertices}, \forall j = 1,...,n, \label{MILP_respect_u} \\
    & (b_{i,j} = 1) \Rightarrow (a_{j,j'})^T x_i \leq w_{j,j'} \quad \forall i = 1,...,p_{vertices}, \forall j,j' = 1,...,n,j < j', \label{MILP_seplin1} \\
    & (b_{i,j'} = 1) \Rightarrow (a_{j,j'})^T x_i \geq w_{j,j'} \quad \forall i = 1,...,p_{vertices}, \forall j,j' = 1,...,n,j < j', \label{MILP_seplin2} \\
    & (a^+_{j,j'})_k \leq s^k_{j,j'} \quad \forall j,j' = 1,...,n,j < j',k = 1,...,m, \label{MILP_sign1} \\
    & (a^-_{j,j'})_k \leq 1-s^k_{j,j'} \quad \forall j,j' = 1,...,n,j < j',k = 1,...,m, \label{MILP_sign2} \\
    & \sum_{k=1}^{m} (a^+_{j,j'})_k + (a^-_{j,j'})_k \geq 1 \quad \forall j,j' = 1,...,n,j < j', \label{MILP_nonzero_vector} \\
    & 0 \leq (a^+_{j,j'})_k, (a^-_{j,j'})_k \leq 1 \quad \forall j,j' = 1,...,n,j < j',k = 1,...,m, \label{MILP_domain_var_sep} \\
    & \sum_{j = 1,...,n} b_{i,j} = 1 \quad \forall i = 1,...,p_{vertices}, \label{MILP_color_each_point} \\
    & \sum_{i = 1,...,p_{vertices}} b_{i,j} \geq 1 \quad \forall j = 1,...,n, \label{MILP_color_each_piece} \\
    & \sum_{i \in e} b_{i,j} \leq |e|-1 \quad \forall e \in E, \label{MILP_edge_description} \\
    & b_{maximalClique_j,j} = 1 \quad \forall j \in [1,|maximalClique|], \label{MILP_fix_binary1} \\
    & b_{i,j} = 0 \quad \forall i \in X \setminus maximalClique_j, \forall j \in [1,|maximalClique|]. \label{MILP_fix_binary0}
\end{align}
\end{subequations}

Constraints (\ref{MILP_respect_l})-(\ref{MILP_respect_u}) are indicator constraints enforcing that if a piece $j$ covers a point $i$, i.e., if $b_{i,j} = 1$, then piece $j$ should satisfy the corridor constraints at that point. 
Similarly, Constraints (\ref{MILP_seplin1})-(\ref{MILP_seplin2}) ensure that if $b_{i,j} = 1$, then point $i$ is in the (domain of) piece $j$.
It is ensured because point $i$ is in the half-plane corresponding to piece $j$ in all linear separations of pieces $j$ and $j'$ by hyperplanes $\{x \in \mathbb R^m \;|\; (a_{j,j'})^T x = w_{j,j'} \}$.
 
Constraints (\ref{MILP_sign1})-(\ref{MILP_domain_var_sep}) ensure that $||a_{j,j'}||_1 \geq 1$ to prevent the vector $a_{j,j'}$ from being the null vector. Indeed, the linear separation requires $a$ to not be the zero vector.
Constraints (\ref{MILP_color_each_point}) ensure that each point is assigned to exactly one piece, while Constraints (\ref{MILP_color_each_piece}) force each piece to cover at least one point so that all pieces are defined.
The constraints (\ref{MILP_edge_description}) are the hypergraph coloring constraints. They prevent that all vertices of an edge $e$ of the incompatibility hypergraph are covered by the same piece. Also, we remove symmetry from the problem by fixing some binary variables $b_{i,j}$. To do that, we use the maximal clique $maximalClique$ to define Constraints \eqref{MILP_fix_binary1} and \eqref{MILP_fix_binary0}. Indeed, each of the vertices $i$ of $maximalClique$ has to be associated to a different piece $j$.

\subsubsection{Algorithm Solving Iteratively the MILP Model}
The algorithm \algmilp{} is given in Algorithm \ref{algorithm_MILP_GC} and iteratively solves the MILP model \eqref{model_DCFP} with an increasing number of vertices $p_{vertices}$ in hypergraph $H$. It produces a lower bound $\mathit{LB}$ of the \RmCFP{}.

Similarly to \algmaximum{}, each iteration of \algmilp{} starts with the computation of the number of vertices $p_{vertices}$ and the incompatibility graph $G=(X,E_2)$. Then unlike \algmaximum{}, a maximal clique of $G$ is computed by the method $\textsc{MaximalCliqueHeuristic}$ described in \ref{subsec:maximalcliqueheuristic}, and used to fix binary variables of the MILP model \eqref{model_DCFP}. If strictly greater than $\mathit{LB}$, the size of the computed clique proves a new lower bound, and the next iteration is immediately started with this new lower bound. Otherwise, a set $E_{3^+}$ of incompatibility edges of size greater or equal to 3 is computed to complete the incompatibility hypergraph $H=(X,E_2 \cup E_{3^+})$. The corresponding MILP model \eqref{model_DCFP} for $LB$ pieces is then solved with the overall remaining time as time limit. If the model is infeasible, then it is impossible to solve this instance of \RmCFP{} with $\mathit{LB}$ pieces; $\mathit{LB}$ is incremented and a new iteration starts. If the model found an optimal solution, then no conclusion can be drawn; a new iteration is started, with more points in $X$ and more incompatibility edges because $k_{optimal}$ is incremented.
Finally if the solver is stopped because of the time limit, no conclusion can be drawn and the while loop stops.
\begin{algorithm}
\textbf{Input:} \\
- the instance of \RmCFP{} $\mathcal{C}=\mathit{Corridor}(u,l,D)$ \\
- an upper bound $\mathit{UB}$ \\
- parameter $k_{repetition}$ of Algorithm \ref{algorithm_build_edge_many} \\
- coefficients $\alpha > 1, p_0, p_{add}$ modeling the evolution of the number of vertices during the algorithm \\
- the computation time limit of the algorithm time\_limit \\ 
- the number of maximal cliques $n_{maximalCliques}$ computed in the function $\textsc{MaximalCliqueHeuristic}$
\\
\textbf{Output:} a lower bound $\mathit{LB}$ of the \RmCFP{} on $\mathit{Corridor}(u,l,D)$ \\
Initialization

\begin{algorithmic}[1]
\State $\mathit{LB} \xleftarrow{} 1$
\State $k_{optimal} \xleftarrow{} 0$
\While{time remaining $> 0$ and $\mathit{LB} < \mathit{UB}$} \label{algo_R1_while}

\noindent \hspace{-0.5cm}Build the incompatibility graph

    \State $p_{vertices} \xleftarrow{} \lceil(p_0 + p_{add} \mathit{LB}) \times \alpha^{k_{optimal}}\rceil$ \label{algo_calcul_N}
    \State $X \xleftarrow{} \textsc{ComputeVertices} (\mathcal{C},p_{vertices})$ \Comment{see Section \ref{subsec:construction_vertices}} \label{algo_calcul_vertices}
    \State $E_2 \xleftarrow{} \textsc{ComputeEdgesOfSize2}(\mathcal{C},X)$ \label{algo_function_aretes2} \Comment{see Section \ref{subsec:construction_aretes_2}}

\noindent \hspace{-0.5cm}Try to tighten the lower bound without solving the MILP if possible

    \State $c \xleftarrow{} \textsc{MaximalCliqueHeuristic}(E_2,n_{maximalCliques})$ \label{algo_heuristique_clique_max} \Comment{see \ref{subsec:maximalcliqueheuristic}}
    \If{$|c| > \mathit{LB}$} \label{algo_test_taille_clique}
        \State $\mathit{LB} \xleftarrow{} |c|$ \label{algo_chgt_val_taille_clique}
    \Else

\noindent \hspace{-0.5cm}Add edges of three vertices or more to build the incompatibility hypergraph
    
        \State $E_3 \xleftarrow{} \textsc{ComputeEdgesOfSize3$^+$}$ $(\mathcal{C},X,E_2,p_{vertices},k_{repetition})$ \label{algo_aretesm} \Comment{see Section \ref{subsec:construction_aretes_multi}}

\noindent \hspace{-0.5cm}Solve the decision version of \RmCFP{} with bound LB
        
        \State $status \xleftarrow{} \textsc{MILP}(\mathcal{C},X,E_2\cup E_3,\mathit{LB},c,\text{time remaining})$ \label{ligne_MILP} \Comment{see Appendix \ref{subsubsec:MILP_model}}

\noindent \hspace{-0.5cm}Increment LB or the vertex parameter
        
        \If{$status$ is infeasible} \label{algo_status_infeasible}
            \State $\mathit{LB} \xleftarrow{} \mathit{LB}+1$ \Comment{increment LB}
        \EndIf
        \If{$status$ is optimal}
            \State $k_{optimal} \xleftarrow{} k_{optimal} + 1$ \Comment{increment the vertex parameter}
        \EndIf \label{algo_status_endif}
    \EndIf
\EndWhile
\State $\text{Return } \mathit{LB}$
\end{algorithmic}
\caption{\algmilp{} based on MILP \eqref{model_DCFP} solving \relaxRonestrong{}}
\label{algorithm_MILP_GC}
\end{algorithm}

An example of run of algorithm \algmilp{} is given below.

\subsubsection{Example of Run of \algmilp{}} \label{subsubsection_example_A1strong}
We describe a run of \algmilp{} with the instance approximating function $L_1$ with absolute approximation error $\delta=1$ and $n_{repetition}=1$.

The algorithm starts with $LB=1$, $UB=6$ and the same parameters as the one used in numerical experiments, described in Section \ref{subsection_instances_params}. The first iteration has a number of vertices $p_{vertices} = 120 (=100+20LB)$. Then, a graph coloring problem with incompatibility graph $(X,E_2)$ is built. Method \textsc{MaximalCliqueHeuristic} finds a maximal clique of 4 vertices, showing that at least 4 colors, and thus 4 linear pieces, must be used in any feasible solution of the corresponding instance of \RCFP{}. Thus, the lower bound is updated to $\mathrm{LB}=4$ and a new while iteration starts. 

The number of vertices for the second iteration is $180 (= 100+20LB)$. A new graph coloring problem is built. As method \textsc{MaximalCliqueHeuristic} finds again a maximal clique of 4 vertices, no better lower bound can be established at this point. Thus method \textsc{ComputeEdgesOfSize3$^+$} produces hyperedges ($179$ out of $180$ trials in this case) to extend the (incompatibility) graph $(X,E_2)$ to a (incompatibility) hypergraph. Then, an instance of MILP model \eqref{model_DCFP} is built and solved via the solver Gurobi. After \SI{0.04}{s}, it is proved that it admits no feasible solution, which means that the hypergraph cannot be colored by 4 colors. By construction, it means that \relaxRonestrong{} has optimal value strictly greater than 4, and that at least 5 pieces are needed in a feasible solution of the instance of \RCFP{} according to Proposition \ref{prop_relaxRone}. The lower bound is thus updated to $\mathrm{LB}=5$.

In the third iteration, the incompatibility graph has $200$ vertices and a maximal clique of 4 which does not improve the lower bound. Thus hyperedges are computed ($200$ out of $200$ trials), and the MILP is built and solved. Gurobi proves that it is infeasible after \SI{200}{s}. This proves that the lower bound is at least $6$, which is also equal to the upper bound. Thus the algorithm stops by returning the lower bound, which is the optimal value for the instance of \RCFP{}.

\section{Sampling the Domain for the $\mathbb{R}^2$-LDRFP} \label{sec:construction_S}
Let $e$ be a subset of vertices of the hypergraph $H$ with at least $3$ vertices. The set $S$ in which are enforced the corridor constraints of $\mathbb R^2$-LDRFP \eqref{eq_R2-DRFP} is constructed using $conv(e)$ as follows. Add $n^S_{edge}$ points distributed uniformly on each side of the polygon $conv(e)$, and add $n^S_{int}$ points inside $conv(e)$. The points inside $conv(e)$ are computed in the following manner. Compute the domain $[x_1^{min},x_1^{max}] \times [x_2^{min},x_2^{max}]$ with $x_i^{min}$ and $x_i^{max}$ the minimal and maximal values of $x_i \in conv(e)$, take samples as elements of a regular grid on this rectangle and count the number of samples that fall inside $conv(e)$. If the number of samples inside is at least equal to $n^S_{int}$, validate those samples. Else, increase the density of the grid, sample and count again.

\section{Improvement of Heuristics of the \RCFP{}: Implementation Details}
\subsection{Improvement of the Piecewise Linear Approximation Method of \citet{Duguet22a}} \label{section_improvement_DLN}
We detail the two improvements to the implementation of Algorithm 1 of \citet{Duguet22a}. In this work, the improved version is refered to as ``DLN'' followed by ``95'' or ``99'' depending on the value of a parameter. 

The first improvement is the change of the Julia package responsible for interval analysis operations. It has been replaced by IntervalArithmetic, resulting in significantly shorter computation times for the whole algorithm because the computations of subrectangles for the PWL corridor is one of the most time-consuming operation. This improvement affects only the computation time.

The second improvement relates to the computation of a linear piece fitting a corridor on a convex subdomain that is ``as large as possible''.
A feasible solution of the maximal piece in direction $d$ problem \citep{Duguet22a} is computed by first building a \PWL{} corridor $\mathcal{C}^{\mathrm{PWL}}$ that is inside the original corridor $\mathcal{C}$, and then by retrieving a linear function inside the corridor $\mathcal{C}^{\mathrm{PWL}}$ as the solution of a linear program. This linear program produces a valid linear function as proved in Proposition 1 of \citet{Duguet22a}. 
We now detail the computation of the \PWL{} corridor on the corridor domain.
A rectangle containing the corridor domain is divided into smaller rectangles iteratively, disregarding the subrectangles of empty intersection with the corridor domain, until all subrectangles satisfy a given approximation level $\eta \in ]0,1[$. This approximation level is called \textit{bounding efficiency} as given in Definition 15 of \citet{Duguet22a}.
The closer $\eta$ is to 1, the better the approximation level is, but also the higher the computation time is because there are more subrectangles. The extreme points of those subrectangles are then used to build the constraints of the linear program. Remark that some of those extreme points may be outside of the corridor domain because it is a general polytope in $\mathbb R^2$ and not a rectangle itself, as shown in Figure \ref{figure_subrectangles}. Thus, the constraints of the linear program coming from the extreme points outside of the corridor domain induce constraints stricter than necessary. The improvement to this process is to produce the linear program from the intersection of all the subrectangles with the corridor domain to remove the unnecessary constraints, as shown in Figure \ref{figure_subrectangles_intersected}. This choice gives more freedom for a linear function to fit the \PWL{} corridor and thus allow to produce feasible solutions of the \RCFP{} with less pieces, but is more time-consuming as there are many subrectangles to intersect at each iteration.

\begin{figure}[!ht]
\centering
\begin{subfigure}{.45\textwidth}
    \centering
    \includegraphics[width=\linewidth]{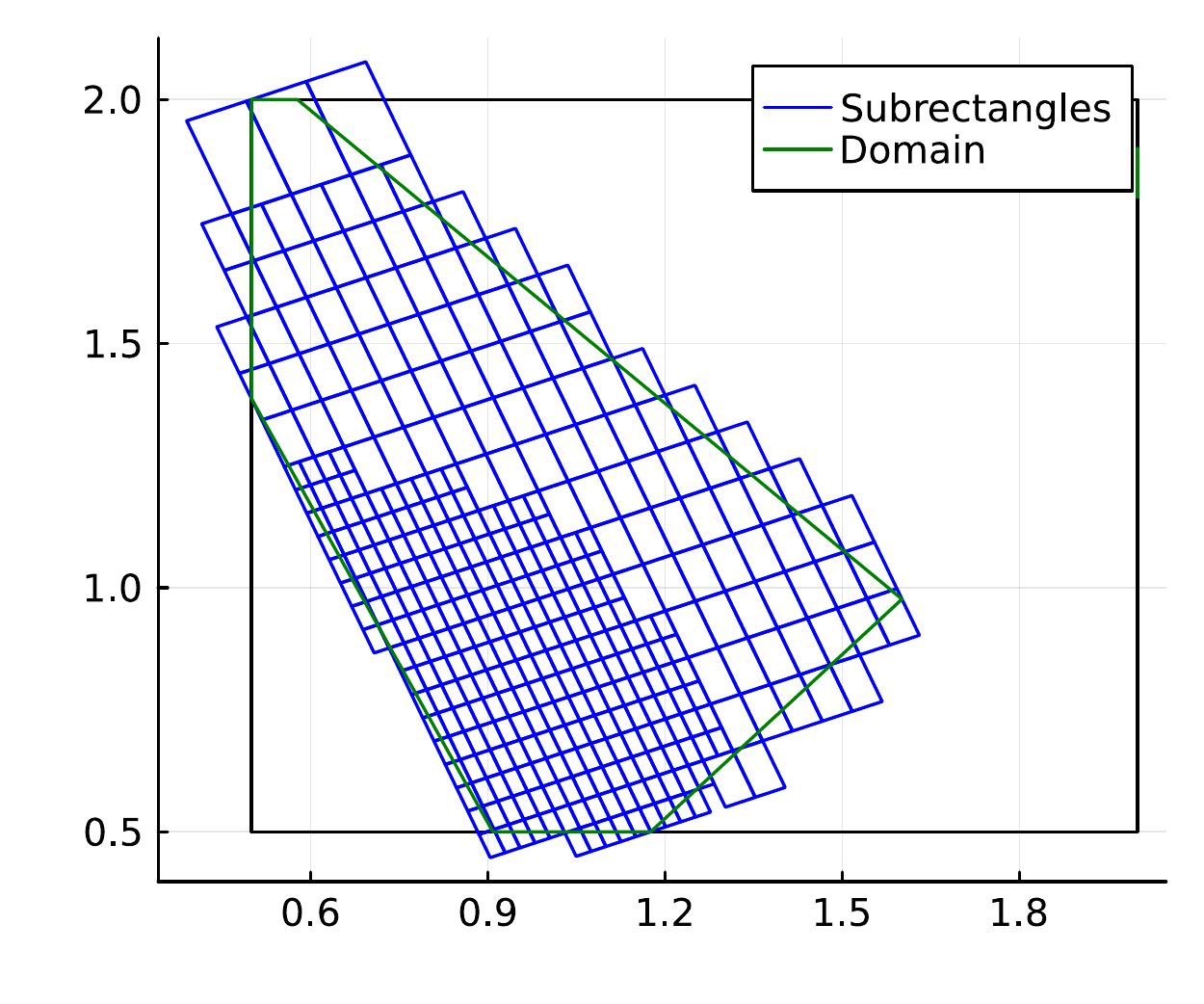}
    \caption{Old version of the PWL corridor from \citet{Duguet22a}}
    \label{figure_subrectangles}
\end{subfigure}%
\begin{subfigure}{.45\textwidth}
    \centering
    \includegraphics[width=\linewidth]{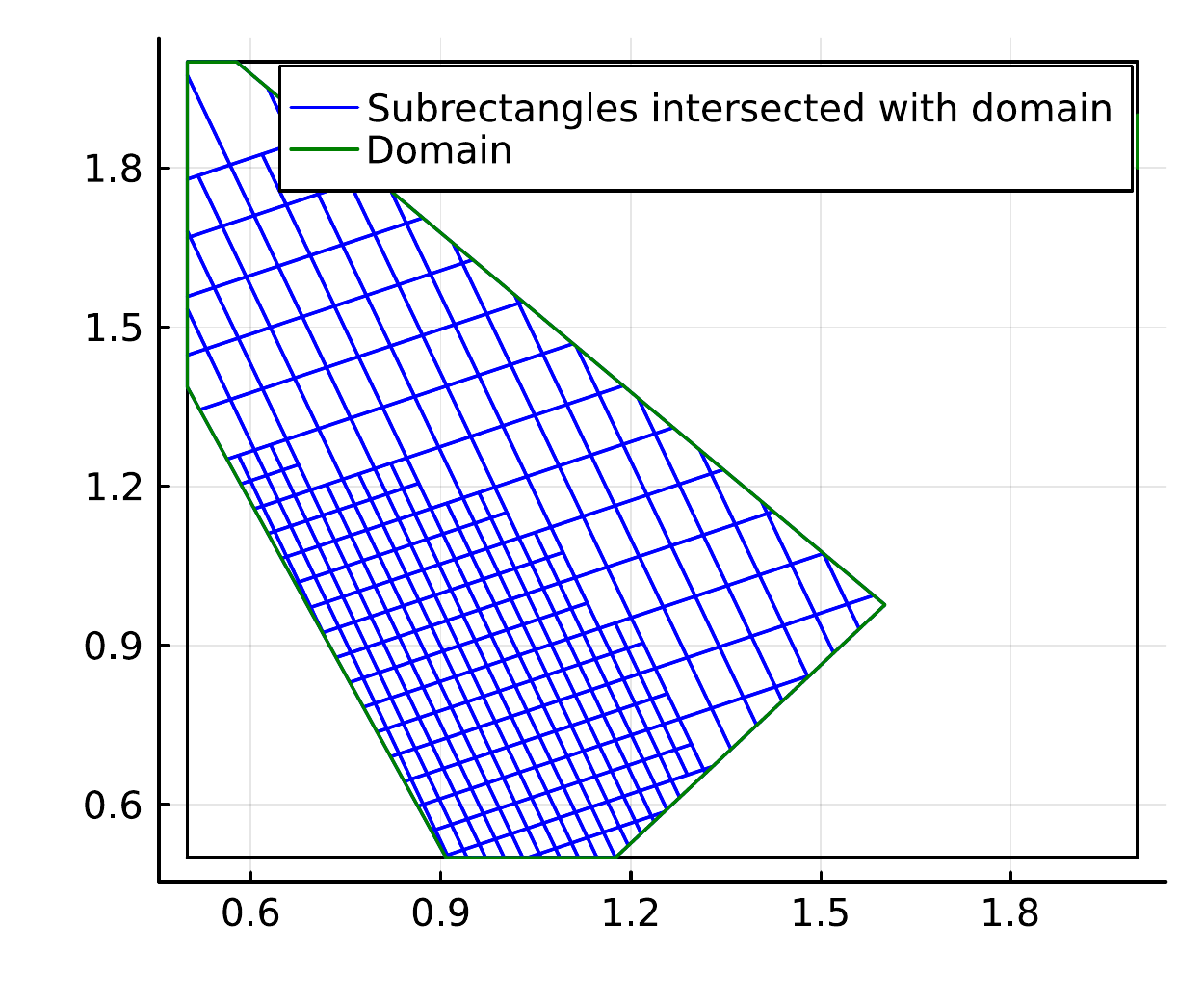}
    \caption{Improved version of the PWL corridor}
    \label{figure_subrectangles_intersected}
\end{subfigure}
\caption{Change in computing the PWL corridor for Algorithm 1 of \citet{Duguet22a}}
\label{figure_improvement_PWL_corridor}
\end{figure}

\subsection{Subroutines used in \iLinA{}: \textsc{MinNumberOfPiecesUnivariate} and \textsc{MinErrorUnivariate}}
\label{section_appendix_findminerror_minnberpieces}

These subroutines are described in Algorithms \ref{algorithm_min_nber_pieces} and \ref{algorithm_iLinA2D_detailed}.

\begin{algorithm}[!ht]
\textbf{Input:} \\
- $f$: the input one-variable function defined on domain $D=[x_{\min},x_{\max}]$\\
- $\delta$: the target precision \\
- PiecewiseLinApprox: official julia package \\ 
\textbf{Output:} $n$ the number of pieces of the piecewise linear approximation \\
Specify error type and value

\begin{algorithmic}[1]
\State err=Absolute($\delta$)
\end{algorithmic}

Compute the optimal piecewise linear approximation with the official julia package PiecewiseLinApprox

\begin{algorithmic}[1]
\setcounter{ALG@line}{1}
\State $g$=PiecewiseLinApprox.Linearize($f$, $x_{\min}$, $x_{\max}$, err, ExactLin(), bounding = Best()) %
\end{algorithmic}

Compute a feasible solution with an equally split error budget
\begin{algorithmic}[1]
\setcounter{ALG@line}{2}
\State $n$=length($g$)

\end{algorithmic}

Return the number of pieces found

\begin{algorithmic}[1]
\setcounter{ALG@line}{3}
\State $\text{Return }$ $n$
\end{algorithmic}
\caption{Algorithm \textsc{MinNumberOfPiecesUnivariate}}
\label{algorithm_min_nber_pieces}
\end{algorithm}

\begin{algorithm}[!ht]
\textbf{Input:} \\
- $f$: the input one-variable function defined on domain $D=[x_{\min},x_{\max}]$\\
- $n_{\mathrm{target}}$: the target number of pieces \\
- $[e_{\min}, e_{\max}]$: the range within which to search for the best error value (by default, $e_{\min} = 0$ and $e_{\max} = \delta$ where $\delta$ is the total error budget of \iLinA{})\\
\textbf{Output:} $e_{\mathrm{best}}$ the best error value \\
Set the numerical precision for the error value
\begin{algorithmic}[1]
\State  $\epsilon = 10^{-3}$ 
\end{algorithmic}

Verify that the best error value can be in the range provided as an input

\begin{algorithmic}[1]
\setcounter{ALG@line}{1}
\State $n_{\min}$=\textsc{MinNumberOfPiecesUnivariate}($f$, $e_{\max}$)  \Comment{see Algorithm \ref{algorithm_min_nber_pieces}}
\State $n_{\max}$=\textsc{MinNumberOfPiecesUnivariate}($f$, $e_{\min}$) 
    \If{$n_{\min} < n_{\mathrm{target}}$ or $n_{\max} > n_{\mathrm{target}}$} 
        Error !
    \EndIf
\end{algorithmic}

Narrow down the interval $[e_{\min}, e_{\max}]$ to a precision $\epsilon$

\begin{algorithmic}[1]
\setcounter{ALG@line}{5}
\While{$e_{\max} - e_{\min} \geq \epsilon$}
    \State $e_{\mathrm{tmp}} = \frac{e_{\min} + e_{\max}}{2}$
    \State $n_{\mathrm{tmp}}$=\textsc{MinNumberOfPiecesUnivariate}($f$, $e_{\mathrm{tmp}}$) 
    \If{$n_{\mathrm{tmp}} == n_{\mathrm{target}}$} 
        \State $e_{\max} = e_{\mathrm{tmp}}$
        \Else
        \State $e_{\min} = e_{\mathrm{tmp}}$
    \EndIf
\EndWhile
\State $e_{\mathrm{best}} = e_{\max} $
\end{algorithmic}

Return the best error value
\begin{algorithmic}[1]
\setcounter{ALG@line}{15}
\State $\text{Return } e_{\mathrm{best}}$ 
\end{algorithmic}
\caption{Algorithm \textsc{MinErrorUnivariate}}
\label{algorithm_iLinA2D_detailed}
\end{algorithm}

\end{document}